\documentclass[11pt]{amsart}
\usepackage[latin1]{inputenc}
\usepackage[T1]{fontenc}
\usepackage{mathtools}
\usepackage{graphicx}
\usepackage{tikz}
\usetikzlibrary{chains}
\usetikzlibrary{cd}
\usepackage{tcolorbox}
\usepackage{mathabx}
\usepackage{enumitem}
\usepackage{tabularx}
\usepackage{float}
\usepackage{longtable}

\PassOptionsToPackage{pdfusetitle,colorlinks}{hyperref}
\usepackage{bookmark}
\hypersetup{
  linkcolor={red!70!black},
  citecolor={green!70!black},
  urlcolor={blue!80!black}
}

\usepackage[
backend=biber,
style=alphabetic,
sorting=nty,
maxcitenames=50,
maxnames=50
]{biblatex}

\renewbibmacro{in:}{}
\usepackage[latin1]{inputenc}
\usepackage{amsmath}
\usepackage{amsthm}
\usepackage{amssymb}

\usepackage[all]{xy}
\usepackage{hyperref}
\newtheorem{theorem}{Theorem}[section]

\newtheorem{proposition}[theorem]{Proposition}

\newtheorem{lemma}[theorem]{Lemma}
\newtheorem{corollary}[theorem]{Corollary}
\newtheorem{conjecture}[theorem]{Conjecture}

\numberwithin{equation}{section}

\setenumerate{label={\rm{(\roman*)}}, leftmargin=*}

\theoremstyle{definition}
\newtheorem{definition}[theorem]{Definition}
\newtheorem{remark}[theorem]{Remark}

\newtheorem{example}[theorem]{Example}
\usepackage[OT2,T1]{fontenc}
\DeclareSymbolFont{cyrletters}{OT2}{wncyr}{m}{n}
\DeclareMathSymbol{\Sha}{\mathalpha}{cyrletters}{"58}

\renewcommand{\to}{\xymatrix@1@=15pt{\ar[r]&}}
\renewcommand{\rightarrow}{\xymatrix@1@=15pt{\ar[r]&}}
\renewcommand{\leftarrow}{\xymatrix@1@=15pt{&\ar[l]}}
\renewcommand{\mapsto}{\xymatrix@1@=15pt{\ar@{|->}[r]&}}
\renewcommand{\twoheadrightarrow}{\xymatrix@1@=18pt{\ar@{->>}[r]&}}
\renewcommand{\hookrightarrow}{\xymatrix@1@=15pt{\ar@{^(->}[r]&}}
\newcommand{\hook}{\xymatrix@1@=15pt{\ar@{^(->}[r]&}}
\newcommand{\congpf}{\xymatrix@L=0.6ex@1@=15pt{\ar[r]^-\sim&}}
\renewcommand{\cong}{\simeq}

\usepackage{comment}
\excludecomment{comment}

\DeclareMathOperator{\im}{im}

\DeclareMathOperator{\Sym}{Sym}

\DeclareMathOperator{\Spec}{Spec}

\DeclareMathOperator{\Aut}{Aut}

\DeclareMathOperator{\PGL}{PGL}
\DeclareMathOperator{\GL}{GL}

\DeclareMathOperator{\Sp}{Sp}

\DeclareMathOperator{\PSL}{PSL}
\DeclareMathOperator{\Res}{Res}
\DeclareMathOperator{\Diag}{diag}

\DeclareMathOperator{\Alt}{Alt}

\renewcommand{\subset}{\subseteq}
\renewcommand{\phi}{\varphi}

\newcommand{\pcoor}[1]{%
  \begingroup\lccode`~=`: \lowercase{\endgroup
  \edef~}{\mathbin{\mathchar\the\mathcode`:}\nobreak}%
  [
  \begingroup
  \mathcode`:=\string"8000
  #1%
  \endgroup 
  ]
}

\newcommand{\ssub}[1]{{Z_{#1}}}
\newcommand{\moduli}[1]{{M_{#1}}}

\title{Special subvarieties in the locus of intermediate Jacobians of cubic threefolds}

\author[Moritz Hartlieb]{Moritz Hartlieb}
\address{Mathematisches Institut, Universitat Bonn, Endenicher Allee 60, 53115
Bonn, Germany}
\email{s6mohart@uni-bonn.de}

\begin{document}

\begin{abstract}
We study special subvarieties, i.e., subvarieties containing a dense subset of CM points, of the moduli space $A_5$ of principally polarized abelian varieties of dimension five, generically contained in the locus of intermediate Jacobians of cubic threefolds. The analogous question for Jacobians of curves is related to a conjecture of Coleman--Ort and has been studied by Shimura, Mostow, De Jong--Noot, Rohde, Moonen, Oort, Frediani, Ghigi and others.

Adapting methods of Frediani, Ghigi and Penegini \cite{frediani2015shimura}, we give a sufficient condition ensuring that the closure of the image of a family of smooth cubic threefolds with prescribed automorphisms via the period map is a special subvariety of $A_5$. 

By the work of Allcock, Carlson and Toledo, it is known that the family of cyclic cubic threefolds gives rise to a special subvariety. 
Analyzing the action of subgroups of the automorphism group of the Klein cubic threefold, we discover new examples of positive-dimensional special subvarieties arising from families of smooth cubic threefolds that are not contained in the locus of cyclic cubic threefolds.
\end{abstract}

\maketitle

\setcounter{section}{-1}
\section{Introduction}

\subsection{} Before giving an overview of the results of this note, let us discuss the analogous question regarding the existence of special subvarieties generically contained in the locus of Jacobians of curves. For details, see the expository article by Moonen \cite{moonen2011torelli}. Let $A_g$ denote the coarse moduli space of principally polarized abelian varieties of dimension $g$.

In \cite[Conj.\ 6]{coleman1987torsion}, Coleman conjectured that for sufficiently large genus, there are only finitely many curves admitting complex multiplication on their Jacobians.
Using the Andr\'e--Oort Conjecture, recently proven in \cite{pila2021canonical}, one can reformulate the conjecture as follows:
\begin{conjecture}[{Coleman--Oort}]
    For $g \geq 8$, there are no positive-dimensional special subvarieties $Z \subset A_g$ generically contained in the Torelli locus.
\end{conjecture}

Without the assumption on the genus, the conjecture fails. In \cite{frediani2015shimura}, Frediani, Ghigi and Penegini establish a sufficient condition ensuring that the image of a family of Galois coverings of the projective line via the period map is a special subvariety of $A_g$, generalizing earlier work by Moonen and others, see \cite{moonen2011torelli}. Using this criterion, they find exactly 30 positive-dimensional special subvarieties of $A_g$ for $g \leq 7$, see \cite[Thm.\ 1.9]{frediani2015shimura}.

As it serves as the main inspiration for large parts of this note, we briefly describe their criterion. A Galois covering $C \to \mathbb P^1$ is determined by the ramification data $m \coloneqq (m_1, \dots, m_r)$, the Galois group $G$, the branching points $t_1, \dots, t_r \in \mathbb P^1$ and an epimorphism $$\theta \colon \pi_1(\mathbb P^1 \setminus \{t_1, \dots, t_r\}, t_0) \twoheadrightarrow G.$$ Fixing the datum 
$(m, G, \theta)$ and varying the points $t_1, \dots, t_r \in \mathbb P^1$, one obtains a family of curves. Let $Z(m, G, \theta)$ denote the closure of the set of Jacobians of these curves in $A_g$. By the Torelli theorem for curves, this an $(r-3)$-dimensional subvariety of $A_g$.
\begin{theorem}[{\cite[Thm.\ 1.4]{frediani2015shimura}}]
\label{thm:frediani}
Let $(m, G, \theta)$ be a datum as above. Assume that
\begin{equation}
    \tag{$\star$} \dim Z(m, G, \theta) = \dim (S^2 H^0(C, K_C))^G. \label{eq:crit}
\end{equation}
Then, $Z(m, G, \theta)$ is a special subvariety of PEL-type of $A_g$, generically contained in the Torelli locus.
\end{theorem}

In \cite{moonen2010special}, Moonen proves that for cyclic $G$ the condition (\ref{eq:crit}) is also necessary. Moreover, he gives a complete classification of those data $(m, \mathbb Z / n \mathbb Z, \theta)$ for which $Z(m, \mathbb Z / n \mathbb Z, \theta)$ is a special subvariety.

\subsection{}
The aim of this note is to study the analogous question for the locus of intermediate Jacobians of cubic threefolds. For a general reference on cubic threefolds, we refer to \cite[Ch.\ 5]{huybrechts23cubic}. Denote by $M \coloneqq H^0(\mathbb P^4, \mathcal O(3))_{\mathrm{sm}} / \GL(5, \mathbb C)$ the coarse moduli space of smooth cubic threefolds.  Similar to the case of curves, the Torelli theorem for cubic threefolds asserts that the map
$$J \colon M \to A_5,$$
sending a cubic threefold to its intermediate Jacobian is a locally closed embedding. A monodromy computation due to Beauville \cite[Thm.\ 4]{beauville2006groupe} shows that the closure of $J(M)$ in $A_5$ is not a special subvariety.\footnote{In \cite{allcock11moduliofcubicthreefoldsasballquotient}, Allcock, Carlson and Toledo give a description of the moduli space $M$ of smooth cubic threefolds as an open subset of a ten-dimensional complex ball quotient via an ''occult'' period map. In particular, there is a notion of special subvarieties of $M$ with respect to this period map. However, note that if the closure of $Z \subset M$ in the ten-dimensional ball quotient is a special subvariety, then, in general, the closure of $J(Z) \subset A_5$ will not be a special subvariety of $A_5$. For example, (the closure of) $M$ is a special subvariety in the ''occult'' sense, but the closure of $J(M)$ in $A_5$ is not a special subvariety.
}

For a finite group $G \subset \GL(5, \mathbb C)$, we let $M_G$ denote the image of $H^0(\mathbb P^4, \mathcal O(3))_{\mathrm{sm}}^G$ in the moduli space of smooth cubic threefolds. In analogy to Theorem \ref{thm:frediani}, we establish the following criterion:
\begin{theorem}[{See Thm.\ \ref{thm:criterion}}]
\label{thm:criterion_intro}
Assume that
\begin{equation}\tag{$\star\star$}\label{eq:crit2}\dim \moduli{G} = \dim (S^2 H^{2, 1}(Y))^G\end{equation}
holds for some smooth cubic threefold $Y \in M_G$. Then, the closure of $J(\moduli{G})$ in $A_5$ is a special subvariety of PEL-type.
\end{theorem}

By the work of Allcock, Carlson and Toledo \cite{allcock2000complex}, it is known that the image of the locus of cyclic cubic threefolds $M^{\mathrm{cyc}} = M_{\langle \Diag(\zeta_3, 1, 1, 1, 1) \rangle} \subset M$ is a special subvariety, see also \cite{achter2012abelian}. Using the classification of groups acting faithfully on smooth cubic threefolds by Wei and Yu \cite{wei2020automorphism}, we show the following:

\begin{theorem}[{See Thm.\ \ref{prop:abelian}}]
    \label{thm:nofurtherexamples}
    Let $G \subset \GL(5, \mathbb C)$ be finite subgroup. If $\dim M_G > 0$ and {\rm{(\ref{eq:crit2})}} is satisfied, then either $M_G$ is contained in the locus of cyclic cubic threefolds $M^{\mathrm{cyc}}$, or $G$ is isomorphic to $\Alt(4)$ or $\Alt(5)$ and $M_G$ contains the Klein cubic threefold.
\end{theorem}
Indeed, the groups $\Alt(4)$ and $\Alt(5)$ give rise to special subvarieties generically contained in the intermediate Jacobian locus:
\begin{theorem}[{See Sec.\ \ref{sec:new_examples}}]
   Up to conjugation, there is a unique subgroup $G \subset \GL(5, \mathbb C)$ isomorphic to  the alternating group $\Alt(4)$ (respectively $\Alt(5)$) for which the closure of $J(M_G)$ in $A_5$ is a special subvariety (of dimension two (respectively one)).

Moreover, $M_G$ contains the Klein cubic threefold and is thus not contained in the locus of cyclic cubic threefolds $M^{\mathrm{cyc}}$.
\end{theorem}


We conclude this introduction by remarking that (\ref{eq:crit2}) is sufficient but not necessary for the closure of $J(\moduli{G})$ to be a special subvariety. In Remark \ref{rem:suffnotnec}, we give examples of groups $G \subsetneq H$, where $G$ does not satify (\ref{eq:crit2}), but $M_G = M_H$ and $H$ satisfies (\ref{eq:crit2}). In particular, the closure of $J(M_G) = J(M_H)$ is a special subvariety.
It would be interesting to know whether there are examples of groups $G \subset \GL(5, \mathbb C)$ for which $\overline{J(M_G)}$ is a special subvariety but there is no subgroup $H \subset \GL(5, \mathbb C)$ with $M_G = M_H$ satisfying (\ref{eq:crit2}).

\subsection{} The plan of this note is as follows:

In Section \ref{sec:gen_special}, we recall the notion special subvarieties and a particular kind of special subvarieties arising from the existence of extra automorphisms.

In Section \ref{sec:extra_aut}, we briefly recall the classification of automorphisms groups of smooth cubic threefolds and liftings of such groups to $\GL(5, \mathbb C)$.

Section \ref{sec:special_in_ij} is devoted to the proof of Theorem \ref{thm:criterion_intro}.

In Section \ref{sec:cyclic}, we discuss the known special subvariety arising from the family of cyclic cubic threefolds.

Section \ref{sec:new_examples} is devoted to the discussion of new examples arising from certain subgroups of the automorphism group of the Klein cubic threefold.

In Section \ref{sec:abelian}, we give the proof of Theorem \ref{thm:nofurtherexamples}.

\subsection*{Acknowledgments}
This work is part of my Master's thesis at the University of Bonn.
I would like to thank my advisor Daniel Huybrechts for his precious comments, advices and suggestions, and Bert van Geemen for his comments on an earlier version of this paper. I am grateful for financial support provided by ERC Synergy Grant HyperK, Grant agreement ID 854361.

\newpage

\section{Special subvarieties}
\label{sec:gen_special}
Let $A_g$ denote the coarse moduli space of $g$-dimensional principally polarized abelian varieties. Recall, e.g., from \cite[Sec.\ 7]{kempf2012complex}, that the Siegel upper half-space
$$\mathbb H_g = \{J \in \GL(2g, \mathbb R) \mid J^2 = -I, J^* E = E, E(x, Jx) > 0, \, \forall x \neq 0\}.$$
parametrizes complex structures on $\mathbb R^{2g} / \mathbb Z^{2g}$, compatible with the principal polarization given by the standard alternating form $E \colon \mathbb Z^{2g} \times \mathbb Z^{2g} \to \mathbb Z$ of type $(1, \dots, 1)$. The group $\Sp(2g, \mathbb Z)$ acts properly discontinuous on $\mathbb H_g$ by conjugation and the quotient $\Sp(2g, \mathbb Z) \backslash \mathbb H_g$ identifies with $A_g$, see \cite[Sec.\ 6]{milne2005introduction}. Denote the principally polarized abelian variety corresponding to a complex structure $J \in \mathbb H_g$ by $(A_J, \Theta_J)$. Observe that $\Aut(A_J, \Theta_J)$ is the stabilizer of the action of $\Sp(2g, \mathbb Z)$ at the point $J \in \mathbb H$.

On $\mathbb H_g$, there is a natural variation of rational Hodge structures, with local system $\mathbb H_g \times \mathbb Q^{2g}$ and corresponding to the Hodge decomposition of $\mathbb C^{2g}$ in $\pm i$ eigenspaces for $J$. By definition, a subvariety $Z \subset A_g$ is a \emph{special subvariety} if it is the image of a Hodge locus of this variation of Hodge structures, see \cite{moonen2011torelli}. 
\begin{theorem}[Andr\'e--Oort, \cite{andre1989g}, \cite{pila2021canonical}]
A subvariety $Z \subset A_g$ is a special subvariety if and only if the set of CM points in $Z$ is a Zariski-dense subset.
\end{theorem}

Let $G$ be a finite subgroup of $\Sp(2g, \mathbb Z)$. One can show that the set of points of $\mathbb H_g$ fixed by $G$ forms a smooth connected submanifold $\mathbb H^G_g \subset \mathbb H_g$, see \cite[Lem.\ 3.3]{frediani2015shimura}. Let $Z_G \subset A_g$ denote the image of $\mathbb H^G_g$ in $A_g$.

\begin{proposition}[{\cite[Prop.\ 3.7]{frediani2015shimura}}]
\label{prop:specialsubfredianighigi}
The subvariety $Z_G \subset A_g$ is a special subvariety (of PEL-type).
\end{proposition}
Alternatively, the subvariety $Z_G \subset A_g$ can be described as the set of complex structures $J \in \mathbb H_g$ for which $G \subset \Aut(A_J, \Theta_J) \subset \Sp(2g, \mathbb Z)$.

One can compute the dimension of $Z_G$ as follows:

\begin{proposition}[{\cite[Lem.\ 3.8]{frediani2015shimura}}]
    \label{prop:specialsubdim}
Let $(A, \Theta)$ be a principally polarized abelian variety corresponding to a point in $\ssub{G} \subset A_g$. Then, we have
$$\dim \ssub{G} = \dim (S^2 H^{0, 1}(A))^G.$$
\end{proposition}
\begin{remark}
Note that $\dim (S^2 H^{0, 1}(A))^G = \dim (S^2 H^{1, 0}(A))^G$.
\end{remark}
As zero-dimensional special subvarieties are nothing but CM points, we obtain the following consequence:
\begin{corollary}[{\cite[Cor.\ 3.10]{frediani2015shimura}}]
Let $(A, \Theta)$ be a principally polarized abelian variety and let $G \subset \Aut(A, \Theta)$ be a group of automorphisms. If
$$\dim (S^2 H^{0, 1}(A))^G = 0,$$
then $A$ admits complex multiplication.
\end{corollary}

\section{Cubic threefolds with extra automorphisms}
\label{sec:extra_aut}

First, recall the description of the coarse moduli space $M$ of smooth cubic threefolds as an affine quotient, see \cite[Ch.\ 3.2]{huybrechts23cubic}: Let  $U \coloneqq H^0(\mathbb P^4, \mathcal O(3))_{\mathrm{sm}} \subset H^0(\mathbb P^4, \mathcal O(3))$ denote the open set of homogeneous cubic polynomials defining smooth cubic threefolds. Note that $U$ is the affine variety $\Spec(A)$, where the ring $A$ is the homogeneous localization of the polynomial ring $\mathbb C[H^0(\mathbb P^4, \mathcal O(3))^\vee]$ with respect to the discriminant $\Delta \in \mathbb C[H^0(\mathbb P^4, \mathcal O(3))^\vee]_{80}$. Then, $M$ is the affine quotient
$$U = \Spec(A) \to M =  \Spec(A^{\GL(5, \mathbb C)}) = U / \GL(5, \mathbb C).$$
 
 \subsection{} Recall that every automorphism of a smooth cubic threefold extends to an automorphism of the ambient projective space, see \cite[Sec.\ 1.3]{huybrechts23cubic}. The groups acting faithfully on smooth cubic threefolds have been classified by Wei and Yu in \cite{wei2020automorphism}, see also \cite{gonzalez2011automorphisms}.
\begin{theorem}[{\cite{wei2020automorphism}}]
\label{thm:autoscubicthreefold}
A group $G$ has a faithful action on some smooth cubic
threefold if and only if $G$ is isomorphic to a subgroup of one of the following six groups: $$(\mathbb Z / 3 \mathbb Z)^4 \rtimes \Sym(5), (((\mathbb Z / 3 \mathbb Z)^2 \rtimes \mathbb Z / 3 \mathbb Z) \rtimes \mathbb Z / 4 \mathbb Z) \times \Sym(3), \mathbb Z / 24 \mathbb Z, \mathbb Z / 16 \mathbb Z, \PSL(2, 11), \mathbb Z / 3 \mathbb Z \times \Sym(5).$$
\end{theorem}
\begin{example}
\label{ex:cubicmaxauto}
The following six smooth cubic threefolds realize the maximal automorphism groups, see \cite[Ex.\ 3.1]{wei2020automorphism}:
{\small
\def\arraystretch{1.2}
\begin{center}
\begin{tabularx}{0.9\textwidth} {  >{\raggedright\arraybackslash}c
  | >{\raggedright\arraybackslash}X 
  | >{\raggedright\arraybackslash}c  }

& $F$ & $\Aut(V(F))$  \\
 \hline
 $Y_1$ & $x_0^3+x_1^3+x_2^3+x_3^3+x_4^3$ & $(\mathbb Z / 3 \mathbb Z)^4 \rtimes \Sym(5)$ \\
\hline
  $Y_2$ & $x_0^3+x_1^3+x_2^3+3(\sqrt{3}-1)x_0x_1x_2+x_3^3+x_4^3$ & $(((\mathbb Z / 3 \mathbb Z)^2 \rtimes \mathbb Z / 3 \mathbb Z) \rtimes \mathbb Z / 4 \mathbb Z) \times \Sym(3)$ \\
\hline
$Y_3$ & $x_0^2x_1 + x_1^2x_2+x_2^2x_3+x_3^3+x_4^3$  & $\mathbb Z / 24 \mathbb Z$ \\
\hline
 $Y_4$ & $x_0^2x_1 + x_1^2x_2+x_2^2x_3+x_3^2x_4+x_4^3$  & $\mathbb Z / 16 \mathbb Z$ \\
\hline
  $Y_5$ &$x_0^2x_1+x_1^2x_2+x_2^2x_3+x_3^2x_4 + x_4^2x_0$ & $\PSL(2, 11)$ \\
\hline
 $Y_6$ & $x_0^3+x_1^2x_2+x_2^2x_3+x_3^2x_4 + x_4^2x_1$   & $\Sym(5) \times \mathbb Z / 3 \mathbb Z$ \\

\end{tabularx}
\def\arraystretch{1}%
\end{center}}
\end{example}

Let us recall some definitions concerning the liftability of group actions on $\mathbb P^n$, following \cite{oguiso2015automorphism}:

\begin{definition}
    Let $F \in \mathbb C[x_0, \dots, x_n]_{d}$ be a homogeneous polynomial of degree $d$ and $H$ a finite subgroup of $\PGL(n+1, \mathbb C).$ A subgroup $G \subset \GL(n+1, \mathbb C)$ is an $F$-\emph{lifting} of $H$ if $G$ and $H$ are isomorphic via the natural projection $\GL(n+1, \mathbb C) \to \PGL(n+1, \mathbb C)$ and $A.F = F$ for all $A \in G$.
\end{definition}

Automorphism groups of cubic threefolds always admit $F$-liftings:
\begin{theorem}[{\cite[Thm.\ 4.11]{wei2020automorphism}}]
\label{thm:fliftings}
    If $Y = V(F) \subset \mathbb P^4$ is a smooth cubic threefold and $H \subseteq \Aut(Y) \subset \PGL(5, \mathbb C)$ is a group of automorphisms of $Y$, then there is an $F$-lifting $G \subseteq \GL(5, \mathbb C)$ of $H$.
\end{theorem}
\begin{remark}
\label{invariantsdontdepend}
    If $3 \mid |H|$, then the $F$-lifting may not be unique. For example, if $A \in \GL(5, \mathbb C)$ is an element of order three then we have $(\zeta_3 A).F = A.F$. However, if $G_1, G_2 \subset \GL(5, \mathbb C)$ are two $F$-liftings of a finite group $H \subset \PGL(5, \mathbb C)$, then  $\langle G_1, \zeta_3 \mathrm{id}_5 \rangle = \langle G_2, \zeta_3 \mathrm{id}_5 \rangle \subset \GL(5, \mathbb C)$ and hence a cubic polynomial is $G_1$-invariant if and only if it is $G_2$-invariant, see \cite[App.\ B]{wei2020automorphism}.
\end{remark}

\subsection{} From now on, we assume that $G \subset \GL(5, \mathbb C)$ is a finite subgroup for which the projection $G \to \PGL(5, \mathbb C)$ is injective. As in the introduction, let $M_G \subset M$ denote the image of $U^G$ in $M = U / GL(5, \mathbb C).$

\begin{lemma}
The subset $M_G \subset M$ is irreducible.
\end{lemma}
\begin{proof}
    Immediately follows from the fact the $U^G$ is irreducible.
\end{proof}

\noindent Note that $M_G$ depends only on the conjugacy class of the subgroup $G \subset \GL(5, \mathbb C)$. For an abstract group $H$, let
$$\widetilde{M}_H \coloneqq \left\{[Y] \in M \mid H \subset \Aut(Y)\right\} \subset M$$
denote the subset of $M$ consisting of all smooth cubic threefolds admitting a faithful action by $H$. By Theorem \ref{thm:fliftings}, we have
$$\widetilde{M}_H = \bigcup_{H \cong G \subset \GL(5, \mathbb C)} M_G.$$
As there are only finitely many subgroups of $\GL(5, \mathbb C)$ isomorphic to $H$ up to conjugation, the subset $\widetilde{M}_H \subset M$ has finitely many irreducible components.

We conclude this section by computing the dimension of $M_G$ in terms of $U^G$ and the centralizer $C_{\GL(5, \mathbb C)}(G)$ of $G$ in $\GL(5, \mathbb C)$.
\begin{lemma}
\label{lem:dim_moduli}
If $M_G \neq \emptyset$, then 
$$\dim \moduli{G} = \dim U^{G}- \dim C_{\GL(5, \mathbb C)}(G).$$
\end{lemma}
\begin{proof}
 Note that the normalizer $N_{\GL(5, \mathbb C)}(G)$ naturally acts on $U^G$ by conjugation. The induced morphism
$$U^G / N_{\GL(5, \mathbb C)}(G) \to U / \GL(5, \mathbb C) = M$$
is finite by \cite[Main Thm.]{luna1975adherence}.
As the stabilizers of the action of the normalizer on $U^G$ are finite, this implies that
$$\dim M_G = \dim U^G - \dim N_{\GL(5, \mathbb C)}(G).$$
Since the quotient $N_{\GL(5, \mathbb C)}(G) / C_{\GL(5, \mathbb C)}(G)$ is isomorphic to a subgroup of the finite group $\Aut(G)$, the claim follows.
\end{proof}

\section{Special subvarieties in the locus of intermediate Jacobians}
\label{sec:special_in_ij}
In this section, we consider special subvarieties generically contained in the locus of intermediate Jacobians of cubic threefolds. Let $Y \subset \mathbb P^4$ be a cubic threefold. Recall that the intermediate Jacobian
$$J(Y) \coloneqq \frac{H^{1, 2}(Y)}{H^3(Y, \mathbb Z)} \cong \frac{H^{2, 1}(Y)^\vee}{H_3(Y, \mathbb Z)}$$
is a principally polarized abelian variety of dimension five. Denote the distinguished theta divisor by $\Xi$. Analogous to the case of curves, there is a Torelli theorem:
\begin{theorem}[Clemens--Griffiths, Tyurin] Let $Y, Y' \subset \mathbb P^4$ be smooth cubic threefolds. Then the following assertions are equivalent:
\begin{enumerate}[leftmargin=*]
    \item There is an isomorphism $Y \cong Y'$.
    \item There is an isomorphism of principally polarized abelian varieties $(J(Y), \Xi) \cong (J(Y'), \Xi')$.
\end{enumerate}
Furthermore, there is a natural isomorphism
    $$\Aut(J(Y), \Xi) \cong \Aut(Y) \times \langle -1 \rangle.$$
\end{theorem}

\begin{proof}
See \cite[]{clemens1972intermediate}. For the claim about automorphism groups, see \cite[Prop.\ 1.6]{zheng2021orbifold}.
\end{proof}

As in the case of curves, the image of the morphism
$$J \colon M \to A_5,$$
sending a smooth cubic hypersurface to the isomorphism class of its intermediate Jacobian, is locally closed.

The aim of this section is to prove Theorem \ref{thm:criterion_intro}. As a crucial input, let us first recall an explicit description of the action of $\Aut(Y)$ on $H^{2, 1}(Y)$.
\begin{lemma}
\label{lem:griffiths}
    Let $G \subset \GL(5, \mathbb C)$ be a finite subgroup and denote the character of the standard representation of $G$ on $\mathbb C^5$ by $\chi$.
    Let $F \in U^G$ be a $G$-invariant cubic polynomial defining a smooth cubic threefold $Y \subset \mathbb P^4$. Then, the character of the natural action of $G$ on $H^{2, 1}(Y)^\vee$ is given by $\det(\chi) \otimes \chi$.
\end{lemma}

\begin{proof}
This is an application of Griffiths' Residue calculus, cf.\ \cite{philip1969periods} and \cite[Sec.\ 3]{beauville2009moduli}.
Let $V \coloneqq \mathbb P^4 \setminus Y$ denote the complement of $Y$. The Gysin exact sequence yields a natural isomorphism
$$\Res \colon H^4(V, \mathbb C) \congpf H^3(Y, \mathbb C).$$
Let $\Omega \coloneqq \sum_{i = 0}^4 (-1)^i x_i d x_0 \wedge \dots \wedge \widehat{dx_i} \wedge \dots \wedge dx_4$. The results of \cite{philip1969periods} imply that the map
\begin{align*}
    H^0(\mathbb P^4, \mathcal O(1)) &\to H^3(Y, \mathbb C) \\
    L &\mapsto \Res \frac{L \Omega}{F^2}
\end{align*}
induces an isomorphism $H^0(\mathbb P^4, \mathcal O(1)) \congpf H^{2, 1}(Y)$. Now use that $F$ is $G$-invariant, that the action of $G$ on $H^0(\mathbb P^4, \mathcal O(1))$ has character $\chi$ and that $G$ acts on $\Omega$ via determinants to conclude the proof.
\end{proof}

In particular, we have
$$\dim (S^2H^{2, 1}(Y))^G = \langle S^2(\det (\chi) \otimes \chi), \chi_{\mathrm{triv}} \rangle,$$
where $\langle\cdot,\cdot\rangle$ denotes the inner product of characters.

\begin{theorem}[{Thm.\ \ref{thm:criterion_intro}}]
\label{thm:criterion}
Let $G \subset \GL(5, \mathbb C)$ be a finite group and let $\chi$ denote the character of the standard representation of $G$ on $\mathbb C^5$. If $M_G \neq \emptyset$ and
\begin{equation}\tag{$\star\star$}\label{eq:crit22}\dim \moduli{G} =  \langle S^2(\det (\chi) \otimes \chi), \chi_{\mathrm{triv}} \rangle,\end{equation}
then the closure of $J(\moduli{G})$ in $A_5$ is a special subvariety (of PEL-type).
\end{theorem}

\begin{proof}
Take a smooth cubic threefold $Y = V(F) \subset \mathbb P^4$ such that $F$ is $G$-invariant.
By the Torelli theorem, $G$ acts faithfully on $H^3(Y, \mathbb Z)$. Fixing an isomorphism $(H^3(Y, \mathbb Z), \cup) \cong (\mathbb Z^{10}, E)$, we may thus identify $G$ with a subgroup of $\Sp(10, \mathbb Z)$. Note that the special subvariety $\ssub{G} \subset A_5$ does not depend on this choice of isomorphism. 

We claim that $J(M_G) \subset Z_G$. Let $F' \in U^G$ be another $G$-invariant cubic polynomial defining a smooth cubic threefold $Y' \subset \mathbb P^4$. Then, there is a path in $U^G$ connecting $F$ and $F'$. Parallel transport along this path gives rise to an isomorphism $(H^3(Y, \mathbb Z), \cup) \cong (H^3(Y', \mathbb Z), \cup)$ such that the diagram
$$\begin{tikzcd}
G \arrow[r, hook] \arrow[rd, hook] & {\Sp(H^3(Y, \mathbb Z), \cup)} \arrow[d, "\cong"] \\
                                   & {\Sp(H^3(Y', \mathbb Z), \cup)}                  
\end{tikzcd}$$
commutes. The claim follows.

Since $Z_G$ is irreducible, it suffices to show that $\dim J(M_G) = \dim Z_G$. By the Torelli theorem, we have $\dim M_G = \dim J(M_G)$, and by \cite[Lem.\ 3.8]{frediani2015shimura}, see Proposition \ref{prop:specialsubdim}, we have
$$\dim Z_G = \dim (S^2H^{2, 1}(Y))^G.$$
As a consequence of Lemma \ref{lem:griffiths}, we have $\dim (S^2H^{2, 1}(Y))^G = \langle S^2(\det(\chi) \otimes \chi), \chi_{\mathrm{triv}} \rangle$. Hence, the condition (\ref{eq:crit22}) implies $\dim J(M_G) = \dim Z_G$ and thus $\overline{J(M_G)} = Z_G$ is a special subvariety.
\end{proof}

\begin{remark}
\label{rem:suffnotnec}
The condition (\ref{eq:crit22}) is only sufficient but not necessary for $\overline{J(M_G)} \subset A_5$ to be a special subvariety. As an example, consider the cyclic group $\mathbb Z / 3 \mathbb Z \cong G \coloneqq \langle \Diag(\zeta_3, \zeta_3, 1, 1, 1) \rangle \subset \GL(5, \mathbb C)$. Then, $\dim M_G = 1$ but $\dim Z_G = 3$. 

However, one can show that, up to coordinate change, every $Y \in M_G$ can be written as
$$V(x_0^3+x_1^3+F(x_2, x_3, x_4)) \subset \mathbb P^4,$$
where $F$ is a homogeneous cubic polynomial. In particular, the group $$\mathbb Z / 3 \mathbb Z \times \mathbb Z / 3 \mathbb Z \cong H \coloneqq \langle \Diag(\zeta_3, 1, 1, 1), \Diag(1, \zeta_3, 1, 1, 1) \rangle \subset \GL(5, \mathbb C)$$ acts on $Y$ and we have $M_G = M_H$. A simple computation shows that $H$ satisfies (\ref{eq:crit22}). Thus, $\overline{J(M_G)} = \overline{J(M_H)}$ is a special subvariety.

It would be interesting to known whether there are examples of groups $G \subset \GL(5, \mathbb C)$ for which $\overline{J(M_G)}$ is a special subvariety but $\overline{J(M_G)} \neq Z_H$ for any $H \subseteq \Sp(10, \mathbb Z)$.
\end{remark}

\begin{corollary}
The intermediate Jacobians of the smooth cubic threefolds $Y_1,\dots, Y_6$ with maximal automorphism group, see Example \ref{ex:cubicmaxauto}, admit complex multiplication.
\end{corollary}
\begin{proof}
For an explicit description of the groups acting on $Y_1, \dots, Y_6$ see \cite[Ex.\ 3.1]{wei2020automorphism}. The case of the Klein cubic threefold $Y_5$ will be recalled in Section \ref{sec:new_examples}. Using \cite{GAP4}, one verifies that in these cases, we have
$$\dim M_G = 0 = \langle S^2(\det(\chi) \otimes \chi), \chi_{\mathrm{triv}}\rangle$$
in the notation of Theorem \ref{thm:criterion}. Hence, the intermediate Jacobians $J(Y_k)$ admit CM.
\end{proof}

\begin{remark}
\label{rem:descy1y5y6}
At least for $Y_1, Y_5$ and $Y_6$, the above is well-known. For $Y_1$ and $Y_6$, we have $J(Y_k) \sim E^5$, where $E$ is the Fermat elliptic curve with complex multiplication by $\mathbb Q(\zeta_3)$, see \cite[Sec.\ 4]{roulleau2009fano} and \cite[Prop.\ 1.7]{van2015intermediate}. For the Klein cubic threefold $Y_5$, Adler has shown in \cite{adler1981some} that $J(Y_5) \sim E_{11}^5$, where $E_{11}$ is an elliptic curve with complex multiplication by $\mathbb Q(\zeta_{11})$. See also \cite{gonzalez2005zerodim}.
\end{remark}

Let us conclude this section by a small refinement of the arguments in the proof of Theorem \ref{thm:criterion}, leading to the observation that (\ref{eq:crit22}) is satisfied if the special subvariety $Z_G$ is generically contained in the intermediate Jacobian locus. 
\begin{proposition}
\label{prop:gencont}
Let $G \subset \Sp(10, \mathbb Z)$ be a finite group. If $Z_G \subset \overline{J(M)}$ and $Z_G \cap J(M) \neq \emptyset$, then there is a finite subgroup $H \subset \GL(5, \mathbb C)$ isomorphic to $\langle G, -1 \rangle / \langle -1 \rangle$ such that $\overline{J(M_H)} = Z_G$. In particular, $H$ satisfies \rm{(\ref{eq:crit22})} of Theorem \ref{thm:criterion}.
\end{proposition}
\begin{proof}
Let $Y = V(F) \subset \mathbb P^4$ be a smooth cubic threefold with $J(Y) \in Z_G$. Then, we have $G \subset \Aut(J(Y), \Xi)$. Let $H$ denote the image of $G$ under the epimorphism
$$\Aut(J(Y), \Xi) \twoheadrightarrow \Aut(Y)$$
described in \cite[Prop.\ 1.6]{zheng2021orbifold}. In particular, we then have $H \cong \langle G, -1 \rangle / \langle -1 \rangle$. By Theorem \ref{thm:fliftings}, we can identify $H$ with a subgroup of $\GL(5, \mathbb C)$ such that $F$ is $H$-invariant. Thus, $Y \in M_H$. By the proof of Theorem \ref{thm:criterion}, we have $M_H \subset Z_G$.

As $Z_G$ is irreducible and, up to conjugation in $\GL(5, \mathbb C)$, there are only finitely many choices for $H$ as above, the claim follows.
\end{proof}

\section{Cyclic cubic threefolds}
\label{sec:cyclic}
In this section, we briefly recall the well-known special subvariety generically contained in the locus of intermediate Jacobians that arises as a family of cyclic cubic threefolds, cf.\ \cite{allcock2000complex} and \cite{carlson2013specialperiods}.

\begin{definition}
A smooth cubic threefold $Y \subset \mathbb P^4$ is called a \emph{cyclic cubic threefold} if there is a cyclic triple cover $Y \to \mathbb P^3$ ramified along a smooth cubic surface.
\end{definition}
\noindent Observe that the above is equivalent to the existence of an automorphism $\phi \in \Aut(Y)$ conjugate to $[\Diag(\zeta_3, 1, 1, 1, 1)] \in \PGL(5, \mathbb C)$. Let
$$M^{\mathrm{cyc}} \coloneqq M_{\langle \Diag(\zeta_3, 1, 1, 1, 1) \rangle} \subset M$$
denote the locus of cyclic cubic threefolds. Let $M^{\mathrm{surf}}$ denote the coarse moduli space of cubic surfaces. The morphism
$$M^{\mathrm{surf}} \to M^{\mathrm{cyc}} \subset M,$$
mapping a cubic surface $S \subset \mathbb P^3$ to a cyclic triple cover $Y \to \mathbb P^3$ ramified along $S$ is generically injective, see \cite[Rem.\ 5.22]{huybrechts23cubic}. Refining this fact, Allcock, Carlson and Toledo have shown that one can embed $M^{\mathrm{surf}}$ in $\overline{J(M^{\mathrm{cyc}})} \subset A_5$, see \cite{allcock2000complex}.

\begin{proposition}
The closure of $J(M^{\mathrm{cyc}})$ in $A_5$ is a special subvariety.
    \label{prop:cycisspecial}
\end{proposition}
\begin{proof}
As discussed in the proof of \cite[Lem.\ 9.2]{allcock2000complex}, the closure of $J(M^{\mathrm{cyc}}) \subset A_5$ is a totally geodesic subvariety,\footnote{See \cite{moonen1998linearity} for a discussion of totally geodesic subvarieties of $A_g$ and their relation to special subvarieties.} isomorphic to a four-dimensional complex ball quotient. By \cite[Thm.\ 4.3]{moonen1998linearity}, a subvariety $Z \subset A_g$ is a special subvariety if and only if it is totally geodesic and contains a CM point. Since the Fermat cubic threefold $Y_1$ is a cyclic cubic threefold and $J(Y_1) \sim E^5$ admits complex multiplication, the claim follows. 

Alternatively, apply Theorem \ref{thm:criterion} to $G = \langle \Diag(\zeta_3, 1, 1, 1, 1) \rangle \subset \GL(5, \mathbb C)$.
\end{proof}

We have the following consequence of Proposition \ref{prop:cycisspecial}:
\begin{corollary}
Let $G \subset \GL(5, \mathbb C)$ be a finite subgroup containing an element $\phi$ conjugate to $\Diag(\zeta_3, 1, 1, 1, 1) \in \GL(5, \mathbb C)$. If $M_G \neq \emptyset$, then the closure of $J(M_G)$ in $A_5$ is a special subvariety.
\end{corollary}

\begin{proof}
In this case, we have $Z_{G} \subset Z_{\langle \phi \rangle} = \overline{J(M^{\mathrm{cyc}})} \subset A_5$. By the arguments given in the proof of Proposition \ref{prop:gencont}, it follows that the closure of $J(M_G)$ is equal to $Z_G$ and thus a special subvariety.
\end{proof}

\section{New examples}
\label{sec:new_examples}
In this section, we discuss new examples of special subvarieties generically contained in the intermediate Jacobian locus arising from families of cubic threefolds admitting an action by the alternating groups $\Alt(4)$ and $\Alt(5)$.

\begin{table}[H]\caption{Irreducible characters of $\PSL(2, 11)$ of degree $\leq 5$}
\label{tab:charpsl211}
\begin{center}
\begin{tabular}{|c|c|c|c|c|c|c|c|c|}
\hline 
 & $1a$ & $2a$ & $3a$ & $5a$ & $5b$ & $6a$ & $11a$ & $11b$ \\
\hline 
$\chi_1$& $1$ & $1$ & $1$ & $1$ & $1$  & $1$ & $1$ & $1$  \\
\hline 
$\chi_2$&  $5$  &  $1$  & $-1$ &  $0$ & $0$ &   $1$ &   $\frac{-1+\sqrt{-11}}{2}$ &   $\frac{-1-\sqrt{11}}{2}$ \\
\hline
$\chi_3$&  $5$  &  $1$  & $-1$ &  $0$ & $0$ &   $1$ &   $\frac{-1-\sqrt{-11}}{2}$ &   $\frac{-1+\sqrt{11}}{2}$ \\
\hline
\end{tabular}
\end{center}
\end{table}
For later use, let us first recall an explicit description of an $F$-lifting of the action of $\PSL(2, 11)$ on the Klein cubic threefold $Y_5 = V(F) \subset \mathbb P^4$, where $$
F \coloneqq x_0x_1^2+x_1x_2^2+x_2x_3^2+x_3x_4^2+x_4x_0^2,$$ see \cite{adler1978automorphism}. The group $\PSL(2, 11)$ has exactly eight conjugacy classes with representatives of order $1, 2, 3, 5, 5, 6, 11$ and $11$. We denote these classes by $1a, 2a, 3a, 5a, 5b, 6a, 11a$ and $11b$. The characters of $\PSL(2, 11)$ of degree at most five are given in Table \ref{tab:charpsl211}. In \cite{adler1978automorphism}, Adler shows that there is a faithful representation $\rho \colon \PSL(2, 11) \to \GL(5, \mathbb C)$ with character $\chi_2$\footnote{
There is an automorphism $\alpha \in \Aut(\PSL(2, 11))$, explicitly given as conjugation by $\Diag(-1, 1) \in \PGL(2, 11)$, that interchanges the conjugacy classes $11a$ and $11b$ and leaves the remaining classes invariant. Hence, if $\rho_2, \rho_3 \colon \PSL(2, 11) \to \GL(5, \mathbb C)$ are five-dimensional representations with characters $\chi(\rho_i) = \chi_i$ for $i = 2, 3$, then their images $\im(\rho_2), \im(\rho_3) \subset \GL(5, \mathbb C)$ are conjugate subgroups. Hence, for our purposes, one could also take a representation with character $\chi_3$.} such that the image $G \coloneqq \im(\rho) \subset \GL(5, \mathbb C)$ is an $F$-lifting of $\Aut(Y_5)$. Using \cite{GAP4}, one computes
$$\dim \moduli{G} = \langle S^3 \chi_2, \chi_1 \rangle - 1 = 0 \quad \text{and} \quad  \dim \ssub{G} = \langle S^2 ( \det(\chi_2) \otimes \chi_2), \chi_1 \rangle = 0,$$confirming the well-known fact that $J(Y_5) \in A_5$ is a special point.

\subsection{}
From now on, we identify $\PSL(2, 11)$ with its image in $\GL(5, \mathbb C)$ via the representation $\rho$ as discussed in the previous section.
Up to isomorphism, the subgroups of $\PSL(2, 11)$ are $\mathbb Z / 2 \mathbb Z$, $\mathbb Z / 3 \mathbb Z$, $(\mathbb Z / 2 \mathbb Z)^2$, $\mathbb Z / 6 \mathbb
 Z$, $\Sym(3)$, $D_{10}$, $\mathbb Z / 11 \mathbb Z$, $\Alt(4), D_{12}$, $\mathbb Z / 11 \mathbb Z \rtimes \mathbb Z / 5 \mathbb Z$, $\Alt(5)$ and $\PSL(2, 11).$
Using the character table of $\PSL(2, 11)$, one easily verifies the following lemma:

\begin{lemma}
If $G_1$ and $G_2$ are subgroups of $\PSL(2, 11) \subset \GL(5, \mathbb C)$ that are isomorphic as abstract groups, then $G_1$ and $G_2$ are conjugate in $\GL(5, \mathbb C)$. In particular, $M_{G_1} = M_{G_2}$.
\end{lemma}

\begin{figure}[H]
\caption{Subgroups of $\PSL(2, 11)$ \label{fig:subgroups}}
\begin{center}{\footnotesize
\newcommand{\groupinfo}[2]{$\begin{matrix}#1 \\ #2\end{matrix}$}
\tikzset{white border/.style={preaction={draw,white,line width=4pt}}}
\begin{tikzpicture}[x=4cm,y=1.5cm,every node/.style={shape=rectangle,draw}]
    \node at (0, 1) (I) {\groupinfo{\{\mathrm{id}\}}{10 < 15}};
    \node at (-1.5, 2) (Z11) {\groupinfo{\mathbb Z / 11 \mathbb Z}{0 = 0}};
    \node at (-0.75, 2) (Z5) {\groupinfo{\mathbb Z / 5 \mathbb Z}{2 < 3}};
    \node at (0, 2) (Z2) {\groupinfo{\mathbb Z / 2 \mathbb Z}{6 < 9}};
    \node at (0.75, 2) (Z3) {\groupinfo{\mathbb Z / 3 \mathbb Z}{4 < 5}};
    \node at (0, 3) (Z2Z2) {\groupinfo{(\mathbb Z / 2 \mathbb Z)^{\times 2}}{4 < 6}};
    \node at (-0.75, 4) (D10) {\groupinfo{D_{10}}{2<3}};
    \node at (0, 4) (A4) {\groupinfo{\Alt(4)}{2 = 2}};
    \node at (0.75, 3) (Z6) {\groupinfo{\mathbb Z / 6 \mathbb Z}{2<3}};
    \node at (1.5, 3) (S3) {\groupinfo{\Sym(3)}{3 < 4}};

     \node at (0, 5) (A5) {\groupinfo{\Alt(5)}{1 = 1}};
    \node at (0.75, 5) (D12) {\groupinfo{D_{12}}{2 < 3}};
   
    \node at (-1.5, 5) (Z11Z5) {\groupinfo{\mathbb Z / 11 \mathbb Z \rtimes \mathbb Z / 5 \mathbb Z}{0 = 0}};
    \node at (0, 6) (PSL) {\groupinfo{\PSL(2, 11)}{0 = 0}};

    \node at (1.5, 6) (MM) {\groupinfo{G}{{\dim \moduli{G}  \leq \dim \ssub{G}}}};

    \draw (I)   -- (Z11);
    \draw (I)   -- (Z5);
    \draw (I)   -- (Z2);
    \draw (I)   -- (Z3);

    \draw (Z3) -- (S3);

     \draw (Z2)   -- (Z2Z2);
     \draw (Z2)   -- (D10);
     \draw (Z2)   -- (Z6);
      \draw (Z2)   -- (S3);

      \draw (Z2Z2) -- (A4);
      \draw (Z2Z2) -- (D12);

      \draw (S3) -- (D12);

    \draw (Z5)   -- (Z11Z5);
    \draw[dashed] (Z11)   -- (Z11Z5);

     \draw[dashed] (Z5)   -- (D10);
     \draw[dashed]  (Z6) -- (D12);

     \draw (A4) -- (A5);
     \draw (D10)   -- (A5);
  
     \draw (A5) -- (PSL);
     \draw(D12) -- (PSL);
    \draw[dashed]  (Z11Z5)   -- (PSL);

    \draw[white border] (Z3) -- (Z6);
    \draw[white border] (Z3) -- (A4);
    \draw[white border] (S3) -- (A5);
\end{tikzpicture}}
\end{center}

\end{figure}
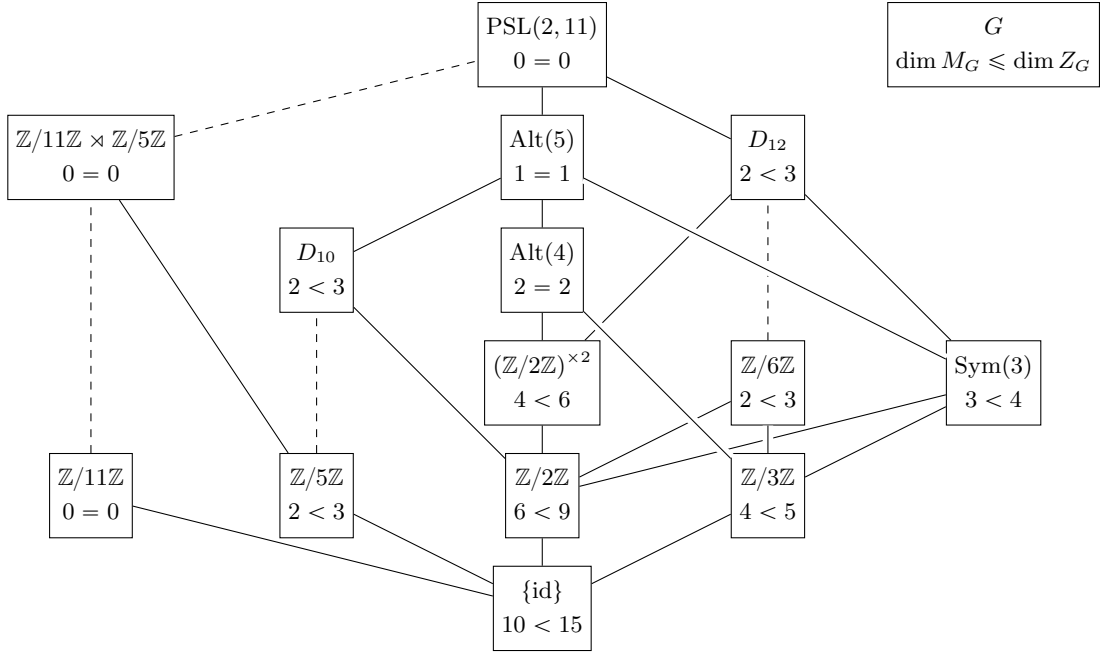

With the help of \cite{GAP4}, we compute $\dim M_G$ and $\dim Z_G$ for all $G \subset \PSL(2, 11)$. The results are listed in Figure \ref{fig:subgroups}, where vertical lines express the subgroup relation. Pairs of subgroups $G \subsetneq G' \subset \PSL(2, 11)$ satisfying $M_G = M_{G'}$ are indicated by dashed lines.

Up to isomorphism, there are exactly two subgroups which satisfy $\dim M_G > 0$ and (\ref{eq:crit22}). These are the alternating groups $\Alt(4)$ and $\Alt(5)$. In the remaining part of this section, we explain the computations and discuss the two resulting special subvarieties.

\subsection{} Let us now explain the computations going into Figure \ref{fig:subgroups} for the subgroup $\Alt(4) \subset \PSL(2, 11)$.
\begin{proposition}
Let $G \subset PSL(2, 11) \subset \GL(5, \mathbb C)$ be a subgroup isomorphic to $\Alt(4)$. Then, the closure of $J(M_G)$ in $A_5$ is a two-dimensional special subvariety. \label{prop:exampleA4}
\end{proposition}

\begin{proof}
The group $\Alt(4)$ has exactly four conjugacy classes $1a, 2a, 3a$ and $3b$ with representatives of order $1, 2, 3$ and $3$. The character table is given in Table \ref{tab:charalt4}. 
The restriction of $$\rho \colon \PSL(2, 11) \to \GL(5, \mathbb C)$$ to $\Alt(4)$ 
has character $\chi \coloneqq \chi_2+\chi_3+\chi_4$. Using \cite{GAP4}, we compute
\begin{equation}\label{eq:alt4}
\dim U^G = \langle S^3 \chi, \chi_\mathrm{triv} \rangle = 5 \quad \text{and} \quad  \langle S^2 ( \det(\chi) \otimes \chi), \chi_\mathrm{triv} \rangle = 2.\end{equation}
\begin{table}[H]\caption{Character table of $\Alt(4)$}
\label{tab:charalt4}
\begin{center}
\begin{tabular}{|c|c|c|c|c|}
\hline 
 & $1a$ & $2a$ & $3a$ & $3b$  \\
\hline 
$\chi_1$& $1$ & $1$ & $1$ & $1$ \\
\hline 
$\chi_2$&  $1$  &  $1$  & $\frac{-1+\sqrt{-3}}{2}$ &  $\frac{-1-\sqrt{-3}}{2}$  \\
\hline
$\chi_3$&  $1$  &  $1$  & $\frac{-1-\sqrt{-3}}{2}$ &  $\frac{-1+\sqrt{-3}}{2}$  \\
\hline
$\chi_4$&  $3$  &  $-1$  & $0$ &  $0$  \\
\hline
\end{tabular}
\end{center}
\end{table}
In view of Lemma \ref{lem:dim_moduli}, it remains to determine the dimension of the centralizer of the image of $\Alt(4)$ in $\GL(5, \mathbb C)$. As a subgroup of the symmetric group $\Sym(4)$, $\Alt(4)$ is generated by the permutations $(243)$ and $(12)(34)$. 
It is easy to verify that
\begin{align*}(243) \mapsto \begin{pmatrix} \zeta_3 & 0 & 0 &0 &0 \\
0 & \zeta_3^2 & 0 & 0 & 0 \\
0 & 0 & 0 & 0 & 1 \\
0 & 0 & 1 & 0 & 0 \\
0 & 0 & 0 & 1 & 0 \end{pmatrix}, \quad
(12)(34) \mapsto \begin{pmatrix}
    1 & 0 & 0 & 0 &0  \\
    0 & 1 & 0 & 0 & 0 \\
    0 & 0 & 1 & 0 & 0 \\
     0 & 0 & 0 & - 1& 0 \\
      0 & 0 & 0 & 0 & -1  \\
\end{pmatrix}
\end{align*}
gives a representation $\Alt(4) \to \GL(5, \mathbb C)$ with character $\chi$. Let $G \subset \GL(5, \mathbb C)$ denote its image. Using the explicit description, one easily checks that the centralizer $C_{\GL(5, \mathbb C)}(G)$ is given by
$$C_{\GL(5, \mathbb C)}(G) = \left\{\Diag(\lambda_1, \lambda_2, \lambda_3, \lambda_3, \lambda_3) \in \GL(5, \mathbb C) \mid \lambda_1, \lambda_2, \lambda_3 \in \mathbb C^\times \right\},$$
and that a general member of $\moduli{G}$ is isomorphic to 
$$V(x_0^3+x_1^3+x_2x_3x_4 + ax_0(x_2^2+\zeta_3^2x_3^2+\zeta_3x_4^2) + bx_1(x_2^2+\zeta_3x_3^2+\zeta_3^2x_4^2)) \subset \mathbb P^4,$$ for some $a, b \in \mathbb C$.
In particular, we have $\dim C_{\GL(5, \mathbb C}(G) = 3$. Combined with
(\ref{eq:alt4}), this yields
$$\dim \moduli{G} = \langle S^3 \chi, \chi_{\mathrm{triv}} \rangle - \dim C_{\GL(5, \mathbb C)}(G) = 2 =  \langle S^2 ( \det(\chi) \otimes \chi), \chi_\mathrm{triv} \rangle = \dim \ssub{G}.$$
We conclude that $\ssub{G} \subset A_5$ is a special subvariety.
\end{proof}

\begin{remark}
By construction, the Klein cubic threefold is contained in the family considered above. Therefore, the family is not contained in the locus of cyclic cubic threefolds. Looking at the explicit equations, we observe that the intersection with the locus of cyclic cubic threefolds is one-dimensional.
\end{remark}

\begin{remark}
\label{rem:desca4}
The locus of smooth cubic threefolds admitting a faithful action by $\Alt(4)$ decomposes as
$$\widetilde{M}_{\Alt(4)} = M_{G_1} \cup M_{G_2},$$
where we may take $\Alt(4) \cong G_1 = G$ as in Proposition \ref{prop:exampleA4} and $\Alt(4) \cong G_2 \subset \GL(5, \mathbb C)$ to be the image of $\Alt(4)$ under the representation
$$\Alt(4) \subset \Sym(4) \to \GL(5, \mathbb C),$$
given by permutation of the first four coordinates. One can show that both the Fermat cubic threefold and the smooth cubic threefold $Y_6$ described in Example \ref{ex:cubicmaxauto} are contained in the intersection $M_{G_1} \cap M_{G_2}$.
\end{remark}

\subsection{} At last, we consider the family of smooth cubic threefolds admitting an action by $\Alt(5)$. The group $\Alt(5)$ has exactly five conjugacy classes  $1a, 2a, 3a, 5a$ and $5b$ with representatives of order $1, 2, 3, 5$ and $5$. The character table is given in Table \ref{tab:charalt5}. 
\begin{table}[H]\caption{Character table of $\Alt(5)$}
\label{tab:charalt5}
\begin{center}
\begin{tabular}{|c|c|c|c|c|c|}
\hline 
 & $1a$ & $2a$ & $3a$ & $5a$ & $5b$  \\
\hline 
$\chi_1$& $1$ & $1$ & $1$ & $1$ & $1$ \\
\hline 
$\chi_2$&  $3$  &  $-1$ & 0 & $\frac{1-\sqrt{5}}{2}$ &  $\frac{1+\sqrt{5}}{2}$  \\
\hline
$\chi_3$&  $3$  &  $-1$  & 0 & $\frac{1+\sqrt{5}}{2}$ &  $\frac{1-\sqrt{5}}{2}$  \\
\hline
$\chi_4$&  $4$  &  $0$ & $1$ & $-1$ &  $-1$  \\
\hline
$\chi_5$&  $5$  &  $1$ & $-1$ & $0$ &  $0$  \\
\hline
\end{tabular}
\end{center}
\end{table}
We have the following description of $\widetilde{M}_{\Alt(5)}$, cf.\ Remark \ref{rem:desca4}:
\begin{lemma}
\label{prop:desca5}
    The locus of of smooth cubic threefolds admitting an action by $\Alt(5)$ has two irreducible components
    $$\widetilde{M}_{\Alt(5)} = M_{H_1} \cup M_{H_2},$$
    where $\Alt(5) \cong H_1 \subset \PSL(2, 11)$ and $H_2$ is the image of $\Alt(5)$ under the representation
    $$\Alt(5) \subset \Sym(5) \to \GL(5, \mathbb C),$$
    given by permutation of coordinates. Moreover, we have $M_{H_1} \cap M_{H_2} = \{Y_1, Y_6\}$.
\end{lemma}
\begin{proof}
We may assume $G_i \subset H_i$ for $i = 1, 2$.

Recall that, by the classification in \cite{gonzalez2011automorphisms}, we have
$\widetilde{M}_{\mathbb Z / 5 \mathbb Z} = M_{\langle \Diag(1, \zeta_5, \zeta_5^2, \zeta_5^3, \zeta_5^4) \rangle}.$ In particular, if $\phi \in \GL(5, \mathbb C)$ is the $F$-lifting of an automorphism of order five of a smooth cubic threefold $Y = V(F) \subset \mathbb P^4$, then $\mathrm{Tr}(\phi) = 0$. Hence, if $H \subset \GL(5, \mathbb C)$ is the $F$-lifting of a group of automorphisms isomorphic to $\Alt(5)$, then the character of the corresponding representation is either given by $\chi_5$ or $\chi_1 + \chi_4$. The character $\chi_5$ corresponds to $H_1$ and $\chi_1 + \chi_4$ corresponds to $H_2$.

If $\Aut(Y)$ contains subgroups isomorphic to $\Alt(5)$ that are not conjugate to each other, then $Y \in \{Y_1, Y_6\}$ by the classification of automorphism groups of cubic threefolds in \cite{wei2020automorphism}.
\end{proof}

Furthermore, we can show that $M_{H_1}$ is exactly the intersection of $M_{G_1}$ and the locus of smooth cubic threefolds admitting an automorphism of order five.
\begin{lemma}
    For $i = 1, 2$, we have
    $$M_{G_i} \cap \widetilde{M}_{\mathbb Z / 5\mathbb Z} = M_{H_i}.$$   
\end{lemma}
\begin{proof}
By considering the groups occuring in the classification of automorphism groups of cubic threefolds given in \cite{wei2020automorphism}, we observe that a cubic threefold $Y$ admits an action by $\Alt(5)$ if and only if it admits an action by $\Alt(4)$ and $\mathbb Z / 5 \mathbb Z$, i.e., we have
$$\widetilde{M}_{\Alt(5)} = \widetilde{M}_{\Alt(4)} \cap \widetilde{M}_{\mathbb Z / 5 \mathbb Z}.$$
By construction, we have $M_{H_i} \subset M_{G_i} \cap \widetilde{M}_{\mathbb Z / 5 \mathbb Z}.$ Towards a contradiction, suppose that $Y$ is a smooth cubic threefold contained in $M_{G_i} \cap \widetilde{M}_{\mathbb Z / 5 \mathbb Z}$ but not in $M_{H_i}$. Then, $\Aut(Y)$ contains a subgroup isomorphic to $\Alt(5)$ and two subgroups isomorphic to $\Alt(4)$ that are not conjugate to each other. By going through the list of automorphism groups of smooth cubic threefolds given in \cite{wei2020automorphism}, this already implies $Y \in \{Y_1, Y_6\}$. But both $Y_1$ and $Y_6$ are contained in $M_{H_i} \cap M_{H_j}$. As desired, this contradicts our assumption on $Y$.
\end{proof}

\begin{proposition}
\label{prop:alt5special}
The closure of $J(M_{H_1})$ in $A_5$ is a one-dimensional special subvariety. Moreover, $M_{H_1}$ contains the Klein cubic threefold and is thus not contained in the locus of cyclic cubic threefolds.
\end{proposition}
\begin{proof}
Since $Z_{H_1} \subset Z_{G_1} = \overline{J(M_{G_1})}$, Proposition \ref{prop:gencont} implies that $\overline{J(M_{H_1})}$ is equal to $Z_{H_1}$ and thus a special subvariety.

It remains to show that $M_{H_1}$ is one-dimensional. The character of the action of $H_1$ on $\mathbb C^5$ equals $\chi_5$. 
Using \cite{GAP4}, we compute
\begin{equation*}\label{eq:alt5}
\dim U^{H_1} = \langle S^3 \chi_5, \chi_\mathrm{triv} \rangle = 2 \quad \text{and} \quad  \langle S^2 ( \det(\chi_5) \otimes \chi_5), \chi_\mathrm{triv} \rangle = 1.\end{equation*}
As the character $\chi_5$ is irreducible, the centralizer of $H_1 \subset \GL(5, \mathbb C)$ is one-dimensional. By Lemma \ref{lem:dim_moduli}, we conclude that $M_{H_1}$ is one-dimensional.
\end{proof}

\begin{remark}The intermediate Jacobians of members of the family $M_{H_1}$ are all isogeneous to the self-product of an elliptic curve:\footnote{Thanks to B.\ van Geemen for pointing this out. In \cite{van2015intermediate}, van Geemen and Yamauchi show that if $Y$ admits an automorphism of order five, i.e., $Y \in \widetilde{M}_{\mathbb Z / 5 \mathbb Z}$, then $J(Y)$ is isogenous to $E \times B^2$, where $E$ is an elliptic curve and $B$ is an abelian surface, and give an  explicit description of $B$ for general $Y \in \widetilde{M}_{\mathbb Z / 5 \mathbb Z}$.} Let $W_5$ denote the five-dimensional irreducible $\mathbb Q$-valued representation of $\Alt(5)$ with character $\chi_5$. The endomorphism algebra of $W_5$ is isomorphic to $\mathbb Q$. By \cite[Thm.\ 13.6.2]{birkehake04cpxav}, this implies that for each cubic threefold $Y \in M_{H_1}$, there is an elliptic curve $E$ such that the self-product $E^5$ is isogeneous to $J(Y)$. Note that this yields an alternative proof of Proposition \ref{prop:alt5special}. See Remark \ref{rem:descy1y5y6} for a description of $E$ for $Y_1, Y_5, Y_6 \in M_{H_1}$.

In particular, the set of intermediate Jacobians  with maximal Picard number in the family $J(M_{H_1})$ is analytically dense, cf.\ \cite[Prop.\ 3 and Prop.\ 4]{beauville14maxpic}.
\end{remark}

\section{Excluding further examples}
\label{sec:abelian}

It turns out that the examples discussed in the two previous sections are the only examples of positive-dimensional special subvarieties generically contained in the intermediate Jacobian locus that arise from the criterion given in Theorem \ref{thm:criterion_intro}.

\begin{theorem}
\label{prop:abelian}
Let $G \subset \GL(5, \mathbb C)$ be a finite subgroup.
If $\dim \moduli{G} > 0$ and $\ssub{G} \subseteq \overline{J(M)}$,\footnote{Here, we identify $G \subset \GL(5, \mathbb C)$ with a subgroup of $\Sp(10, \mathbb Z)$ via the inclusion $G \subset \Aut(Y) \congpf \Aut(J(Y)) \subset \Sp(H^3(Y, \mathbb Z)) \cong \Sp(10, \mathbb Z)$ for some $Y \in M_G$. As in the proof of Theorem \ref{thm:criterion}, we note that $Z_G \subset A_5$ only depends on $G \subset \GL(5, \mathbb C)$ and not on the choice of $Y$ or the isomorphism $H^3(Y, \mathbb Z) \cong \mathbb Z^{10}$.} then either $M_G \subset M^{\mathrm{cyc}}$, or $G$ is isomorphic to $\Aut(4)$ or $\Aut(5)$ and $M_G$ contains the Klein cubic threefold.
\end{theorem}

\begin{remark}
    In the latter case, $M_G$ is one of the families described in the previous section, see Remark \ref{rem:desca4} and Proposition \ref{prop:desca5}.
\end{remark}

\begin{remark}
    The condition $\dim M_G > 0$ is necessary: For example, the cubic threefold $Y_4$ with $\Aut(Y_4) \cong \mathbb Z / 16 \mathbb Z$ is not a cyclic cubic threefold but admits complex multiplication on its intermediate Jacobian. Similarly, the Klein cubic threefold $Y_5$ is not a cyclic cubic threefold, but any cyclic subgroup $G \subset \Aut(Y_5) \cong \PSL(2, 11)$ of order eleven satisfies $\dim M_{G} = 0$ and (\ref{eq:crit22}). Note that the cubic threefolds $Y_1, Y_2, Y_3$ and $Y_6$ in Example \ref{ex:cubicmaxauto} are cyclic cubic threefolds.
    \end{remark}

\begin{proof}[Proof of Theorem \ref{prop:abelian}]
Recall from the proof of Theorem \ref{thm:criterion} that $\overline{J(M_G)} \subset Z_G$ with equality if and only (\ref{eq:crit22}) holds. By Proposition \ref{prop:gencont}, the condition that $Z_G$ is contained in  $\overline{J(M)}$ already implies that $G$ satisfies (\ref{eq:crit22}).

By \cite{wei2020automorphism}, the groups admitting a faithful action on a smooth cubic threefold are precisely the subgroups of the six groups listed in Theorem \ref{thm:autoscubicthreefold}. The strategy of this proof is to use \cite{GAP4} to list all five-dimensional representations of these groups and to check when these satisfy (\ref{eq:crit2}).

First, let us fix some notation. Let $H$ be a subgroup of one of the six groups described in Theorem \ref{thm:autoscubicthreefold}. Using \cite{GAP4}, we obtain a complete list of characters that correpond to representations $H \to \GL(5, \mathbb C)$ for which the composition $H \to \GL(5, \mathbb C) \to \PGL(5, \mathbb C)$ is injective.
In the following, let $\chi$ be such a character, $\rho \colon H \to \GL(5, \mathbb C)$ a representation with character $\chi$, and let $G \subset \GL(5, \mathbb C)$ denote its image. Since, up to conjugation in $\GL(5, \mathbb C)$, the representation $\rho$ is determined by $\chi$, the subset $M_G$ only depends on the character $\chi$.

One can compute the dimension of the centralizer of $G$ in $\GL(5, \mathbb C)$ using the following elementary lemma: 
\begin{lemma}
    \label{lem:dim_centralizer}
    If $\chi = \sum n_i \chi_i$ is a decomposition of $\chi$ into irreducible characters, then
    $$\dim C_{\GL(5, \mathbb C)}(G) = \sum_{i} n_i^2.$$
\end{lemma}

If the group $G$ is abelian, then every irreducible character is of degree one. Moreover, in \cite[App.\ B]{wei2020automorphism}, Wei and Yu give a list of possible $F$-liftings of abelian groups acting faithfully on 
smooth cubic threefolds. Hence, we can use \cite{GAP4} to compute $\dim M_G$ and $\dim Z_G$ for all abelian groups acting on smooth cubic threefolds. The results are listed in Table \ref{tab:ab}.

In the case that we have $\dim M_G = \dim Z_G$ in Table \ref{tab:ab}, the subgroup $G \subset \GL(5, \mathbb C)$ contains an element conjugate to $\Diag(\zeta_3, 1, 1, 1, 1)$ except for the families No.\ 43, 48, 54, 58, 59, 61, 63, 66 and 67. By determining the corresponding sets $U^G$ of $G$-invariant homogeneous cubic polynomials, one can show that $M_G \subset M^{\mathrm{cyc}}$ holds also in these cases:
\begin{lemma}
    \label{lem:abelian_families_cyclic}
    The families No.\ $43, 48, 54, 58, 59, 61, 63, 66, 67$ in Table \ref{tab:ab} are contained in $M^{\mathrm{cyc}}$.
\end{lemma}
\begin{proof}
It is easy to verify that in these cases, every $G$-invariant homogeneous cubic polynomial is conjugate to a polynomial of the form $x_0^3+F(x_1, \dots, x_4)$.

For example, consider Family No.\ 43.  Then, we have
$$G \coloneqq  \langle \Diag(1, \zeta_{3}^{1}, 1,1,  \zeta_{3}^{1}), \Diag(1, 1, \zeta_{3}^{1},\zeta_{3}^{2},  \zeta_{3}^{2}) \rangle \subset \GL(5, \mathbb C).$$
One easily checks that we have
$$H^0(\mathbb P^4, \mathcal O(3))^G = \langle x_0^3, x_1^3, x_2^3, x_3^3, x_4^3, x_0x_2x_3 \rangle \subset H^0(\mathbb P^4, \mathcal O(3))$$
and thus $M_G \subset M^{\mathrm{cyc}}$.
The remaining cases are left to the reader.
\end{proof}
This finishes the proof in the case of abelian groups. 

In order to carry out the computations in the non-abelian case, let us recall a few criteria for determining whether a subgroup $G \subset \GL(5, \mathbb C)$
satisfies $\dim M_G \neq \emptyset$.

\begin{lemma}
\label{lem:criteria_autocubic}
Let $G \subset \GL(5, \mathbb C)$ be a subgroup for which the projection $G \to \PGL(5, \mathbb C)$ is injective and $M_G \neq \emptyset.$ Let $\phi \in G$ be an element of order $n$.
\begin{itemize}
\item If $n = 2$, then $\phi$ is conjugate to 
$\Diag(-1, 1, 1, 1, 1) \text{ or } \Diag(-1, -1, 1, 1, 1).$

\item If $n = 4$, then $\phi$ is not conjugate to
    $\Diag(\zeta_4, 1, 1, 1, 1) \text{ and } \Diag(-\zeta_4, 1, 1, 1, 1).$

\item If $n = 5$, then $\phi$ is conjugate to
    $\Diag(1, \zeta_5, \zeta_5^2, \zeta_5^3, \zeta_5^4).$
\end{itemize}
\end{lemma}
\begin{proof}
Follows from \cite[Tab.\ 2]{wei2020automorphism}.
\end{proof}

For every non-abelian group $H$ occuring in the classification of automorphism groups of smooth cubic threefolds, we use \cite{GAP4} to create a list of all characters of $H$ of degree five. 
Then, we apply Lemma \ref{lem:dim_centralizer}, Lemma \ref{lem:criteria_autocubic} and the dimension formulas given in Proposition \ref{prop:specialsubdim} and Lemma \ref{lem:dim_moduli} to exclude groups $H$ for which (\ref{eq:crit2}) can not be satisfied. Furthermore, we remove the characters $\chi$ for which the images of the corresponding representations contain an element conjugate to a scalar multiple of $\Diag(\zeta_3, 1, 1, 1, 1)$ from our list, because the associated families of cubic threefolds are contained in the locus of cyclic cubic threefolds. The groups that survive this process are $\mathbb Z / 3 \mathbb Z \rtimes \mathbb Z / 4 \mathbb Z$, $\mathbb Z / 3 \mathbb Z \times \Sym(3)$, $(\mathbb Z / 3 \mathbb Z \times \mathbb Z / 3 \mathbb Z) \rtimes \mathbb Z / 3 \mathbb Z$,  $((\mathbb Z / 3 \mathbb Z \times \mathbb Z / 3 \mathbb Z) \rtimes \mathbb Z / 3 \mathbb Z) \rtimes \mathbb Z / 2 \mathbb Z$, $\Alt(4)$ and $\Alt(5)$.

In the remaining part of this proof, we show that only the alternating groups $\Alt(4)$ and $\Alt(5)$ admit representations for which $\dim M_G > 0$, $M_G \not \subseteq M^{\mathrm{cyc}}$ and $\dim M_G = \dim Z_G$.

\begin{lemma}
    If $G \subset \GL(5, \mathbb C)$ is isomorphic to $\mathbb Z / 3 \mathbb Z \rtimes \mathbb Z / 4 \mathbb Z$, then $\dim M_G \neq \dim Z_G$.
\end{lemma}
\begin{proof}
    The group $\mathbb Z / 3 \mathbb Z \rtimes \mathbb Z / 4 \mathbb Z$ has exactly six conjugacy classes $1a, 4a, 2a,  3a, 4b$ and  $6a$ with representatives of order $1, 4, 2, 3, 4$ and $6$. The character table is given in Table \ref{tab:charZ3Z4}. \begin{table}[H]\caption{Character table of $\mathbb Z / 3 \mathbb Z  \rtimes \mathbb Z / 4 \mathbb Z$}
        \label{tab:charZ3Z4}
        \begin{center}
        \begin{tabular}{|c|c|c|c|c|c|c|}
        \hline 
         & $1a$ & $4a$ & $2a$ & $3a$ & $4b$ & $6a$ \\
        \hline 
        $\chi_1$& $1$ & $1$ & $1$ & $1$ & $1$ & $1$   \\
        \hline 
        $\chi_2$& $1$ & $-1$ & $1$ & $1$ & $-1$ & $1$ \\
        \hline
        $\chi_3$& $1$ & $-i$ & $-1$ & $1$ & $-i$ & $-1$ \\
        \hline
        $\chi_4$& $1$ & $i$ & $-1$ & $1$ & $i$ & $-1$ \\
        \hline
        $\chi_5$& $2$ & $0$ & $-2$ & $-1$ & $0$ & $1$ \\
        \hline
        $\chi_6$& $2$ & $0$ & $2$ & $-1$ & $0$ & $-1$ \\
        \hline 
        \end{tabular}
        \end{center}
        \end{table}
        Going through all characters of degree five and applying the process described above, the characters $2 \chi_1 + \chi_3 + \chi_6$ and $2 \chi_1 + \chi_4 + \chi_6$ are the only possible candidates for which (\ref{eq:crit22}) could hold. Let us consider the case $\chi \coloneqq 2 \chi_1 + \chi_4 + \chi_6$.
        One checks that
        $$(0, 1) \mapsto \begin{pmatrix}
            1 & 0 & 0 & 0 &0  \\
            0 & 1 & 0 & 0 & 0 \\
            0 & 0 &  \zeta_4 & 0 & 0 \\
             0 & 0 & 0 &0 & 1 \\
              0 & 0 & 0 & 1 & 0  \\
        \end{pmatrix}, \quad (1, 0) \mapsto \begin{pmatrix}
            1 & 0 & 0 & 0 &0  \\
            0 & 1 & 0 & 0 & 0 \\
            0 & 0 & 1 & 0 & 0 \\
             0 & 0 & 0 &\zeta_3 & 0 \\
              0 & 0 & 0 & 0 & \zeta_3^2  \\
        \end{pmatrix}$$
        gives rise to a representation $\mathbb Z / 3 \mathbb Z \rtimes \mathbb Z / 4 \mathbb Z \to \GL(5, \mathbb C)$ with character $\chi$. Let $G \subset \GL(5, \mathbb C)$ denote its image. The space of $G$-invariant homogeneous cubic polynomials is of dimension $$\dim H^0(\mathbb P^4, \mathcal O(3))^G = \langle S^3 \chi, \chi_{\mathrm{triv}} \rangle = 7$$ and spanned by
        $$H^0(\mathbb P^4, \mathcal O(3))^G = \langle x_0^3, x_0^2 x_1, x_0 x_1^2, x_1^3, x_0 x_3 x_4, x_1 x_3 x_4, x_3^3+x_4^3 \rangle \subset H^0(\mathbb P^4, \mathcal O(3)).$$
        As the variable $x_2$ does not occur among these polynomials, we conclude that $M_G = \emptyset$.

        The character $2 \chi_1 + \chi_3 + \chi_6$ is excluded using similar arguments.
\end{proof}

\begin{lemma}
If $G \subset \GL(5, \mathbb C)$ is isomorphic to one of
$$\mathbb Z / 3 \mathbb Z \times \Sym(3), (\mathbb Z / 3 \mathbb Z \times \mathbb Z / 3 \mathbb Z) \rtimes \mathbb Z / 3 \mathbb Z,  ((\mathbb Z / 3 \mathbb Z \times \mathbb Z / 3 \mathbb Z) \rtimes \mathbb Z / 3 \mathbb Z) \rtimes \mathbb Z / 2 \mathbb Z,
$$
and $\dim M_G = \dim Z_G$, then $M_G \subset M^{\mathrm{cyc}}$.
\end{lemma}
\begin{proof}
Let us consider the group $G \cong \mathbb Z / 3 \mathbb Z  \times \Sym(3)$. The remaining cases are left to the reader.  The group  $\mathbb Z / 3 \mathbb Z  \times \Sym(3)$ has exactly nine conjugacy classes with representatives $1a, 2a, 3a, 3b, 6a, 3c, 3d, 6b$ and $3e$ of order $1, 2, 3, 3, 6, 3, 3, 6$ and $3$. The character table is given in Table \ref{tab:charZ3S3}. 

\begin{table}[H]\caption{Character table of $\mathbb Z / 3 \mathbb Z  \times \Sym(3)$}
    \label{tab:charZ3S3}
    \begin{center}
    \begin{tabular}{|c|c|c|c|c|c|c|c|c|c|}
    \hline 
     & $1a$ &  $2a$ &  $3a$ &  $3b$ &   $6a$ & $3c$ &   $3d$ & $6b$ &  $3e$ \\
    \hline 
    $\chi_1$& $1$ & $1$ & $1$ & $1$ & $1$ & $1$ & $1$ & $1$ & $1$  \\
    \hline 
    $\chi_2$& $1$ & $-1$ &   $1$  &  $1$ &   $-1$ &  $1$ &  $1$  & $-1$ &  $1$ \\
    \hline
    $\chi_3$& $1$ &  $-1$ &   $\zeta_3^2$  &  $1$  &  $-\zeta_3^2$ &  $\zeta_3$ &  $\zeta_3^2$ & $-\zeta_3$ & $\zeta_3$ \\
    \hline
    $\chi_4$& $1$ &  $-1$ &  $\zeta_3$ &   $1$ &  $-\zeta_3$ &  $\zeta_3^2$ &  $\zeta_3$ &  $-\zeta_3^2$ &  $\zeta_3^2$ \\
    \hline
    $\chi_5$& $1$  & $1$  & $\zeta_3^2$ &  $1$ &  $\zeta_3^2$ & $\zeta_3$ &  $\zeta_3^2$ &  $\zeta_3$ &  $\zeta_3$ \\
    \hline
    $\chi_6$& $1$ &  $1$ &  $\zeta_3$ &  $1$ & $\zeta_3$ & $\zeta_3^2$ &  $\zeta_3$ &   $\zeta_3^2$ &   $\zeta_3^2$ \\
    \hline 
    $\chi_7$ & $2$ & $0$ &  $2$ &  $-1$ &    $0$ &   $2$ &   $-1$ &    $0$ & $-1$ \\
    \hline
    $\chi_8$ &  $2$  & $0$ &   $2\zeta_3$ & $-1$ &    $0$ &  $2\zeta_3^2$ &  $-\zeta_3$ &    $0$  &  $-\zeta_3^2$ \\
    \hline
    $\chi_9$ &  $2$  & $0$ &   $2 \zeta_3^2$ & $-1$ &    $0$ &  $2 \zeta_3$ &  $-\zeta_3^2$ &    $0$  &  $-\zeta_3$ \\
    \hline
    \end{tabular}
    \end{center}
    \end{table}
Going through the list of characters of degree five and applying the process described above, we see that the only characters we have to check are those of the form $\chi_i + \chi_j + \chi_k$, where
$$(i, j, k) \in \{(1, 7, 9), (5, 7, 9), (1, 7, 8), (6, 7, 8), (5, 8, 9), (6, 8, 9)\}.$$
A straightforward computation as in Lemma \ref{lem:abelian_families_cyclic} shows that for these characters, we have $M_G \subset M^{\mathrm{cyc}}$. In fact, $M_G$ agrees with the one-dimensional family No.\ 32 in Table \ref{tab:ab}.
\end{proof}

We conclude that if $\dim M_G > 0$ and (\ref{eq:crit2}) are satified, then either $G \cong \Alt(4), \Alt(5)$ and $M_G$ contains the Klein cubic threefold, or $M_G$ is contained in the locus of cyclic cubic threefolds. This finishes the proof of Theorem \ref{prop:abelian}.
\end{proof}

{\footnotesize
\def\arraystretch{1.5}
\begin{longtable}{|c|c|p{5cm}|c|c|c|}
\caption{Abelian groups acting on smooth cubic threefolds (cf.\ \cite{wei2020automorphism})}
\label{tab:ab}
\\
\hline
No. &$H\subset \Aut(X)$  & generator(s) of an $F$-lifting $G$ of $H$  & $\dim M_G$ & $\dim Z_G$ & $=$ \\
\hline
1 & $\mathbb Z / {2} \mathbb Z$ & $\Diag(-1, 1, 1,1,  1)$ & 7 &  11 & No  \\ \hline
2 & $\mathbb Z / {2} \mathbb Z$ & $\Diag(-1, -1, 1,1,  1)$ & 6 &  9 & No  \\ \hline
3 & $\mathbb Z / {3} \mathbb Z$ & $\Diag(1, 1, 1,\zeta_{3}^{1},  1)$ & 4 &  4 & Yes  \\ \hline
4 & $\mathbb Z / {3} \mathbb Z$ & $\Diag(1, 1, 1,\zeta_{3}^{1},  \zeta_{3}^{1})$ & 1 &  3 & No  \\ \hline
5 & $\mathbb Z / {3} \mathbb Z$ & $\Diag(1, 1, 1,\zeta_{3}^{1},  \zeta_{3}^{2})$ & 4 &  7 & No  \\ \hline
6 & $\mathbb Z / {3} \mathbb Z$ & $\Diag(1, 1, \zeta_{3}^{1},\zeta_{3}^{1},  \zeta_{3}^{2})$ & 4 &  5 & No  \\ \hline
7 & $\mathbb Z / {4} \mathbb Z$ & $\Diag(\zeta_{4}^{1}, -1, 1,1,  1)$ & 3 &  4 & No  \\ \hline
8 & $\mathbb Z / {4} \mathbb Z$ & $\Diag(\zeta_{4}^{1}, -1, 1,-1,  1)$ & 3 &  5 & No  \\ \hline
9 & $\mathbb Z / {4} \mathbb Z$ & $\Diag(\zeta_{4}^{1}, -1, 1,\zeta_{4}^{3},  1)$ & 3 &  5 & No  \\ \hline
10 & $(\mathbb Z / {2} \mathbb Z)^{2}$ & $\Diag(-1, 1, 1,1,  1)$, $\Diag(1, 1, -1,1,  1)$ & 5 &  8 & No  \\ \hline
11 & $(\mathbb Z / {2} \mathbb Z)^{2}$ & $\Diag(-1, -1, 1,1,  1)$, $\Diag(1, -1, -1,1,  1)$ & 4 &  6 & No  \\ \hline
12 & $\mathbb Z / {5} \mathbb Z$ & $\Diag(1, \zeta_{5}^{1}, \zeta_{5}^{2},\zeta_{5}^{3},  \zeta_{5}^{4})$ & 2 &  3 & No  \\ \hline
13 & $\mathbb Z / {2} \mathbb Z \times \mathbb Z / {3} \mathbb Z$ & $\Diag(-1, 1, -1,1,  1)$, $\Diag(1, 1, 1,\zeta_{3}^{1},  1)$ & 2 &  2 & Yes  \\ \hline
14 & $\mathbb Z / {2} \mathbb Z \times \mathbb Z / {3} \mathbb Z$ & $\Diag(-1, 1, -1,1,  1)$, $\Diag(1, 1, \zeta_{3}^{1},\zeta_{3}^{1},  1)$ & 1 &  2 & No  \\ \hline
15 & $\mathbb Z / {2} \mathbb Z \times \mathbb Z / {3} \mathbb Z$ & $\Diag(-1, 1, -1,1,  1)$, $\Diag(1, 1, \zeta_{3}^{1},\zeta_{3}^{1},  \zeta_{3}^{1})$ & 1 &  2 & No  \\ \hline
16 & $\mathbb Z / {2} \mathbb Z \times \mathbb Z / {3} \mathbb Z$ & $\Diag(-1, 1, -1,1,  1)$, $\Diag(1, 1, \zeta_{3}^{1},\zeta_{3}^{1},  \zeta_{3}^{2})$ & 2 &  3 & No  \\ \hline
17 & $\mathbb Z / {2} \mathbb Z \times \mathbb Z / {3} \mathbb Z$ & $\Diag(-1, 1, 1,1,  1)$, $\Diag(1, 1, \zeta_{3}^{1},1,  1)$ & 3 &  3 & Yes  \\ \hline
18 & $\mathbb Z / {2} \mathbb Z \times \mathbb Z / {3} \mathbb Z$ & $\Diag(-1, 1, 1,1,  1)$, $\Diag(1, 1, \zeta_{3}^{1},1,  \zeta_{3}^{1})$ & 1 &  3 & No  \\ \hline
19 & $\mathbb Z / {2} \mathbb Z \times \mathbb Z / {3} \mathbb Z$ & $\Diag(-1, 1, 1,1,  1)$, $\Diag(1, 1, \zeta_{3}^{1},1,  \zeta_{3}^{2})$ & 3 &  5 & No  \\ \hline
20 & $\mathbb Z / {2} \mathbb Z \times \mathbb Z / {3} \mathbb Z$ & $\Diag(-1, 1, 1,1,  1)$, $\Diag(1, 1, \zeta_{3}^{1},\zeta_{3}^{1},  \zeta_{3}^{1})$ & 1 &  2 & No  \\ \hline
21 & $\mathbb Z / {2} \mathbb Z \times \mathbb Z / {3} \mathbb Z$ & $\Diag(-1, 1, 1,1,  1)$, $\Diag(1, 1, \zeta_{3}^{1},\zeta_{3}^{1},  \zeta_{3}^{2})$ & 2 &  3 & No  \\ \hline
22 & $\mathbb Z / {2} \mathbb Z \times \mathbb Z / {3} \mathbb Z$ & $\Diag(-1, 1, 1,1,  1)$, $\Diag(1, 1, \zeta_{3}^{1},\zeta_{3}^{2},  \zeta_{3}^{2})$ & 2 &  3 & No  \\ \hline
23 & $\mathbb Z / {8} \mathbb Z$ & $\Diag(\zeta_{8}^{1}, \zeta_{8}^{6}, \zeta_{8}^{4},1,  1)$ & 1 &  2 & No  \\ \hline
24 & $\mathbb Z / {8} \mathbb Z$ & $\Diag(\zeta_{8}^{1}, \zeta_{8}^{6}, \zeta_{8}^{4},1,  \zeta_{8}^{2})$ & 1 &  2 & No  \\ \hline
25 & $\mathbb Z / {4} \mathbb Z \times \mathbb Z / {2} \mathbb Z$ & $\Diag(\zeta_{4}^{1}, -1, 1,1,  1)$, $\Diag(1, 1, 1,-1,  1)$ & 2 &  3 & No  \\ \hline
26 & $\mathbb Z / {9} \mathbb Z$ & $\Diag(\zeta_{9}^{1}, \zeta_{9}^{7}, \zeta_{9}^{4},1,  1)$ & 0 &  0 & Yes  \\ \hline
27 & $\mathbb Z / {9} \mathbb Z$ & $\Diag(\zeta_{9}^{1}, \zeta_{9}^{7}, \zeta_{9}^{4},1,  \zeta_{9}^{3})$ & 0 &  1 & No  \\ \hline
28 & $\mathbb Z / {9} \mathbb Z$ & $\Diag(\zeta_{9}^{1}, \zeta_{9}^{7}, \zeta_{9}^{4},1,  \zeta_{9}^{6})$ & 0 &  1 & No  \\ \hline
29 & $\mathbb Z / {9} \mathbb Z$ & $\Diag(\zeta_{9}^{1}, \zeta_{9}^{7}, \zeta_{9}^{4},\zeta_{9}^{3},  \zeta_{9}^{3})$ & 0 &  0 & Yes  \\ \hline
30 & $\mathbb Z / {9} \mathbb Z$ & $\Diag(\zeta_{9}^{1}, \zeta_{9}^{7}, \zeta_{9}^{4},\zeta_{9}^{3},  \zeta_{9}^{6})$ & 0 &  1 & No  \\ \hline
31 & $\mathbb Z / {9} \mathbb Z$ & $\Diag(\zeta_{9}^{1}, \zeta_{9}^{7}, \zeta_{9}^{4},\zeta_{9}^{6},  \zeta_{9}^{6})$ & 0 &  0 & Yes  \\ \hline
32 & $(\mathbb Z / {3} \mathbb Z)^{2}$ & $\Diag(1, \zeta_{3}^{1}, 1,1,  1)$, $\Diag(1, 1, \zeta_{3}^{1},1,  1)$ & 1 &  1 & Yes  \\ \hline
33 & $(\mathbb Z / {3} \mathbb Z)^{2}$ & $\Diag(1, \zeta_{3}^{1}, 1,1,  1)$, $\Diag(1, 1, \zeta_{3}^{1},1,  \zeta_{3}^{1})$ & 0 &  0 & Yes  \\ \hline
34 & $(\mathbb Z / {3} \mathbb Z)^{2}$ & $\Diag(1, \zeta_{3}^{1}, 1,1,  1)$, $\Diag(1, 1, \zeta_{3}^{1},1,  \zeta_{3}^{2})$ & 2 &  2 & Yes  \\ \hline
35 & $(\mathbb Z / {3} \mathbb Z)^{2}$ & $\Diag(1, \zeta_{3}^{1}, 1,1,  1)$, $\Diag(1, 1, \zeta_{3}^{1},\zeta_{3}^{1},  \zeta_{3}^{1})$ & 1 &  1 & Yes  \\ \hline
36 & $(\mathbb Z / {3} \mathbb Z)^{2}$ & $\Diag(1, \zeta_{3}^{1}, 1,1,  1)$, $\Diag(1, 1, \zeta_{3}^{1},\zeta_{3}^{1},  \zeta_{3}^{2})$ & 2 &  2 & Yes  \\ \hline
37 & $(\mathbb Z / {3} \mathbb Z)^{2}$ & $\Diag(1, \zeta_{3}^{1}, 1,1,  1)$, $\Diag(1, 1, \zeta_{3}^{1},\zeta_{3}^{2},  \zeta_{3}^{2})$ & 2 &  2 & Yes  \\ \hline
38 & $(\mathbb Z / {3} \mathbb Z)^{2}$ & $\Diag(1, \zeta_{3}^{1}, 1,1,  \zeta_{3}^{1})$, $\Diag(1, 1, \zeta_{3}^{1},1,  1)$ & 0 &  0 & Yes  \\ \hline
39 & $(\mathbb Z / {3} \mathbb Z)^{2}$ & $\Diag(1, \zeta_{3}^{1}, 1,1,  \zeta_{3}^{1})$, $\Diag(1, 1, \zeta_{3}^{1},1,  \zeta_{3}^{1})$ & 0 &  1 & No  \\ \hline
40 & $(\mathbb Z / {3} \mathbb Z)^{2}$ & $\Diag(1, \zeta_{3}^{1}, 1,1,  \zeta_{3}^{1})$, $\Diag(1, 1, \zeta_{3}^{1},1,  \zeta_{3}^{2})$ & 0 &  1 & No  \\ \hline
41 & $(\mathbb Z / {3} \mathbb Z)^{2}$ & $\Diag(1, \zeta_{3}^{1}, 1,1,  \zeta_{3}^{1})$, $\Diag(1, 1, \zeta_{3}^{1},\zeta_{3}^{1},  \zeta_{3}^{1})$ & 0 &  1 & No  \\ \hline
42 & $(\mathbb Z / {3} \mathbb Z)^{2}$ & $\Diag(1, \zeta_{3}^{1}, 1,1,  \zeta_{3}^{1})$, $\Diag(1, 1, \zeta_{3}^{1},\zeta_{3}^{1},  \zeta_{3}^{2})$ & 0 &  1 & No  \\ \hline
43 & $(\mathbb Z / {3} \mathbb Z)^{2}$ & $\Diag(1, \zeta_{3}^{1}, 1,1,  \zeta_{3}^{1})$, $\Diag(1, 1, \zeta_{3}^{1},\zeta_{3}^{2},  \zeta_{3}^{2})$ & 1 &  1 & Yes  \\ \hline
44 & $(\mathbb Z / {3} \mathbb Z)^{2}$ & $\Diag(1, \zeta_{3}^{1}, 1,1,  \zeta_{3}^{2})$, $\Diag(1, 1, \zeta_{3}^{1},1,  1)$ & 2 &  2 & Yes  \\ \hline
45 & $(\mathbb Z / {3} \mathbb Z)^{2}$ & $\Diag(1, \zeta_{3}^{1}, 1,1,  \zeta_{3}^{2})$, $\Diag(1, 1, \zeta_{3}^{1},1,  \zeta_{3}^{1})$ & 0 &  1 & No  \\ \hline
46 & $(\mathbb Z / {3} \mathbb Z)^{2}$ & $\Diag(1, \zeta_{3}^{1}, 1,1,  \zeta_{3}^{2})$, $\Diag(1, 1, \zeta_{3}^{1},1,  \zeta_{3}^{2})$ & 1 &  3 & No  \\ \hline
47 & $(\mathbb Z / {3} \mathbb Z)^{2}$ & $\Diag(1, \zeta_{3}^{1}, 1,1,  \zeta_{3}^{2})$, $\Diag(1, 1, \zeta_{3}^{1},\zeta_{3}^{1},  \zeta_{3}^{1})$ & 0 &  1 & No  \\ \hline
48 & $(\mathbb Z / {3} \mathbb Z)^{2}$ & $\Diag(1, \zeta_{3}^{1}, 1,1,  \zeta_{3}^{2})$, $\Diag(1, 1, \zeta_{3}^{1},\zeta_{3}^{1},  \zeta_{3}^{2})$ & 2 &  2 & Yes  \\ \hline
49 & $(\mathbb Z / {3} \mathbb Z)^{2}$ & $\Diag(1, \zeta_{3}^{1}, 1,1,  \zeta_{3}^{2})$, $\Diag(1, 1, \zeta_{3}^{1},\zeta_{3}^{2},  \zeta_{3}^{2})$ & 2 &  3 & No  \\ \hline
50 & $(\mathbb Z / {3} \mathbb Z)^{2}$ & $\Diag(1, \zeta_{3}^{1}, 1,\zeta_{3}^{1},  \zeta_{3}^{1})$, $\Diag(1, 1, \zeta_{3}^{1},1,  1)$ & 1 &  1 & Yes  \\ \hline
51 & $(\mathbb Z / {3} \mathbb Z)^{2}$ & $\Diag(1, \zeta_{3}^{1}, 1,\zeta_{3}^{1},  \zeta_{3}^{1})$, $\Diag(1, 1, \zeta_{3}^{1},1,  \zeta_{3}^{1})$ & 0 &  1 & No  \\ \hline
52 & $(\mathbb Z / {3} \mathbb Z)^{2}$ & $\Diag(1, \zeta_{3}^{1}, 1,\zeta_{3}^{1},  \zeta_{3}^{1})$, $\Diag(1, 1, \zeta_{3}^{1},1,  \zeta_{3}^{2})$ & 0 &  1 & No  \\ \hline
53 & $(\mathbb Z / {3} \mathbb Z)^{2}$ & $\Diag(1, \zeta_{3}^{1}, 1,\zeta_{3}^{1},  \zeta_{3}^{1})$, $\Diag(1, 1, \zeta_{3}^{1},\zeta_{3}^{1},  \zeta_{3}^{1})$ & 0 &  1 & No  \\ \hline
54 & $(\mathbb Z / {3} \mathbb Z)^{2}$ & $\Diag(1, \zeta_{3}^{1}, 1,\zeta_{3}^{1},  \zeta_{3}^{1})$, $\Diag(1, 1, \zeta_{3}^{1},\zeta_{3}^{1},  \zeta_{3}^{2})$ & 1 &  1 & Yes  \\ \hline
55 & $(\mathbb Z / {3} \mathbb Z)^{2}$ & $\Diag(1, \zeta_{3}^{1}, 1,\zeta_{3}^{1},  \zeta_{3}^{1})$, $\Diag(1, 1, \zeta_{3}^{1},\zeta_{3}^{2},  \zeta_{3}^{2})$ & 0 &  1 & No  \\ \hline
56 & $(\mathbb Z / {3} \mathbb Z)^{2}$ & $\Diag(1, \zeta_{3}^{1}, 1,\zeta_{3}^{1},  \zeta_{3}^{2})$, $\Diag(1, 1, \zeta_{3}^{1},1,  1)$ & 2 &  2 & Yes  \\ \hline
57 & $(\mathbb Z / {3} \mathbb Z)^{2}$ & $\Diag(1, \zeta_{3}^{1}, 1,\zeta_{3}^{1},  \zeta_{3}^{2})$, $\Diag(1, 1, \zeta_{3}^{1},1,  \zeta_{3}^{1})$ & 0 &  1 & No  \\ \hline
58 & $(\mathbb Z / {3} \mathbb Z)^{2}$ & $\Diag(1, \zeta_{3}^{1}, 1,\zeta_{3}^{1},  \zeta_{3}^{2})$, $\Diag(1, 1, \zeta_{3}^{1},1,  \zeta_{3}^{2})$ & 2 &  2 & Yes  \\ \hline
59 & $(\mathbb Z / {3} \mathbb Z)^{2}$ & $\Diag(1, \zeta_{3}^{1}, 1,\zeta_{3}^{1},  \zeta_{3}^{2})$, $\Diag(1, 1, \zeta_{3}^{1},\zeta_{3}^{1},  \zeta_{3}^{1})$ & 1 &  1 & Yes  \\ \hline
60 & $(\mathbb Z / {3} \mathbb Z)^{2}$ & $\Diag(1, \zeta_{3}^{1}, 1,\zeta_{3}^{1},  \zeta_{3}^{2})$, $\Diag(1, 1, \zeta_{3}^{1},\zeta_{3}^{1},  \zeta_{3}^{2})$ & 2 &  3 & No  \\ \hline
61 & $(\mathbb Z / {3} \mathbb Z)^{2}$ & $\Diag(1, \zeta_{3}^{1}, 1,\zeta_{3}^{1},  \zeta_{3}^{2})$, $\Diag(1, 1, \zeta_{3}^{1},\zeta_{3}^{2},  \zeta_{3}^{2})$ & 1 &  1 & Yes  \\ \hline
62 & $(\mathbb Z / {3} \mathbb Z)^{2}$ & $\Diag(1, \zeta_{3}^{1}, 1,\zeta_{3}^{2},  \zeta_{3}^{2})$, $\Diag(1, 1, \zeta_{3}^{1},1,  1)$ & 2 &  2 & Yes  \\ \hline
63 & $(\mathbb Z / {3} \mathbb Z)^{2}$ & $\Diag(1, \zeta_{3}^{1}, 1,\zeta_{3}^{2},  \zeta_{3}^{2})$, $\Diag(1, 1, \zeta_{3}^{1},1,  \zeta_{3}^{1})$ & 1 &  1 & Yes  \\ \hline
64 & $(\mathbb Z / {3} \mathbb Z)^{2}$ & $\Diag(1, \zeta_{3}^{1}, 1,\zeta_{3}^{2},  \zeta_{3}^{2})$, $\Diag(1, 1, \zeta_{3}^{1},1,  \zeta_{3}^{2})$ & 2 &  3 & No  \\ \hline
65 & $(\mathbb Z / {3} \mathbb Z)^{2}$ & $\Diag(1, \zeta_{3}^{1}, 1,\zeta_{3}^{2},  \zeta_{3}^{2})$, $\Diag(1, 1, \zeta_{3}^{1},\zeta_{3}^{1},  \zeta_{3}^{1})$ & 0 &  1 & No  \\ \hline
66 & $(\mathbb Z / {3} \mathbb Z)^{2}$ & $\Diag(1, \zeta_{3}^{1}, 1,\zeta_{3}^{2},  \zeta_{3}^{2})$, $\Diag(1, 1, \zeta_{3}^{1},\zeta_{3}^{1},  \zeta_{3}^{2})$ & 1 &  1 & Yes  \\ \hline
67 & $(\mathbb Z / {3} \mathbb Z)^{2}$ & $\Diag(1, \zeta_{3}^{1}, 1,\zeta_{3}^{2},  \zeta_{3}^{2})$, $\Diag(1, 1, \zeta_{3}^{1},\zeta_{3}^{2},  \zeta_{3}^{2})$ & 2 &  2 & Yes  \\ \hline
68 & $\mathbb Z / {11} \mathbb Z$ & $\Diag(\zeta_{11}^{1}, \zeta_{11}^{9}, \zeta_{11}^{4},\zeta_{11}^{3},  \zeta_{11}^{5})$ & 0 &  0 & Yes  \\ \hline
69 & $\mathbb Z / {4} \mathbb Z \times \mathbb Z / {3} \mathbb Z$ & $\Diag(\zeta_{4}^{1}, -1, 1,1,  1)$, $\Diag(1, 1, 1,\zeta_{3}^{1},  1)$ & 1 &  1 & Yes  \\ \hline
70 & $\mathbb Z / {4} \mathbb Z \times \mathbb Z / {3} \mathbb Z$ & $\Diag(\zeta_{4}^{1}, -1, 1,1,  1)$, $\Diag(1, 1, 1,\zeta_{3}^{1},  \zeta_{3}^{1})$ & 0 &  0 & Yes  \\ \hline
71 & $\mathbb Z / {4} \mathbb Z \times \mathbb Z / {3} \mathbb Z$ & $\Diag(\zeta_{4}^{1}, -1, 1,1,  1)$, $\Diag(1, 1, 1,\zeta_{3}^{1},  \zeta_{3}^{2})$ & 1 &  2 & No  \\ \hline
72 & $\mathbb Z / {4} \mathbb Z \times \mathbb Z / {3} \mathbb Z$ & $\Diag(\zeta_{4}^{1}, -1, 1,1,  -1)$, $\Diag(1, 1, 1,\zeta_{3}^{1},  \zeta_{3}^{1})$ & 0 &  1 & No  \\ \hline
73 & $\mathbb Z / {4} \mathbb Z \times \mathbb Z / {3} \mathbb Z$ & $\Diag(\zeta_{4}^{1}, -1, 1,1,  \zeta_{4}^{3})$, $\Diag(1, 1, 1,\zeta_{3}^{1},  1)$ & 1 &  1 & Yes  \\ \hline
74 & $(\mathbb Z / {2} \mathbb Z)^{2} \times \mathbb Z / {3} \mathbb Z$ & $\Diag(-1, 1, 1,1,  1)$, $\Diag(1, 1, -1,1,  1)$, $\Diag(1, 1, \zeta_{3}^{1},\zeta_{3}^{1},  1)$ & 1 &  2 & No  \\ \hline
75 & $(\mathbb Z / {2} \mathbb Z)^{2} \times \mathbb Z / {3} \mathbb Z$ & $\Diag(-1, 1, 1,1,  1)$, $\Diag(1, 1, -1,1,  1)$, $\Diag(1, 1, \zeta_{3}^{1},\zeta_{3}^{1},  \zeta_{3}^{1})$ & 1 &  2 & No  \\ \hline
76 & $(\mathbb Z / {2} \mathbb Z)^{2} \times \mathbb Z / {3} \mathbb Z$ & $\Diag(-1, 1, 1,1,  1)$, $\Diag(1, 1, -1,1,  1)$, $\Diag(1, 1, \zeta_{3}^{1},\zeta_{3}^{1},  \zeta_{3}^{2})$ & 1 &  2 & No  \\ \hline
77 & $(\mathbb Z / {2} \mathbb Z)^{2} \times \mathbb Z / {3} \mathbb Z$ & $\Diag(-1, 1, 1,1,  1)$, $\Diag(1, 1, -1,1,  1)$, $\Diag(1, 1, 1,\zeta_{3}^{1},  1)$ & 2 &  2 & Yes  \\ \hline
78 & $(\mathbb Z / {2} \mathbb Z)^{2} \times \mathbb Z / {3} \mathbb Z$ & $\Diag(-1, 1, 1,1,  -1)$, $\Diag(1, 1, -1,1,  -1)$, $\Diag(1, 1, 1,\zeta_{3}^{1},  1)$ & 1 &  1 & Yes  \\ \hline
79 & $\mathbb Z / {5} \mathbb Z \times \mathbb Z / {3} \mathbb Z$ & $\Diag(\zeta_{5}^{1}, \zeta_{5}^{3}, \zeta_{5}^{4},\zeta_{5}^{2},  1)$, $\Diag(1, 1, 1,1,  \zeta_{3}^{1})$ & 0 &  0 & Yes  \\ \hline
80 & $\mathbb Z / {16} \mathbb Z$ & $\Diag(\zeta_{16}^{1}, \zeta_{16}^{14}, \zeta_{16}^{4},\zeta_{16}^{8},  1)$ & 0 &  0 & Yes  \\ \hline
81 & $\mathbb Z / {9} \mathbb Z \times \mathbb Z / {2} \mathbb Z$ & $\Diag(\zeta_{9}^{1}, \zeta_{9}^{7}, \zeta_{9}^{4},1,  1)$, $\Diag(1, 1, 1,-1,  1)$ & 0 &  0 & Yes  \\ \hline
82 & $\mathbb Z / {9} \mathbb Z \times \mathbb Z / {2} \mathbb Z$ & $\Diag(\zeta_{9}^{1}, \zeta_{9}^{7}, \zeta_{9}^{4},\zeta_{9}^{3},  \zeta_{9}^{3})$, $\Diag(1, 1, 1,-1,  1)$ & 0 &  0 & Yes  \\ \hline
83 & $\mathbb Z / {9} \mathbb Z \times \mathbb Z / {2} \mathbb Z$ & $\Diag(\zeta_{9}^{1}, \zeta_{9}^{7}, \zeta_{9}^{4},\zeta_{9}^{6},  \zeta_{9}^{6})$, $\Diag(1, 1, 1,-1,  1)$ & 0 &  0 & Yes  \\ \hline
84 & $\mathbb Z / {2} \mathbb Z \times (\mathbb Z / {3} \mathbb Z)^{2}$ & $\Diag(-1, 1, 1,1,  1)$, $\Diag(1, 1, \zeta_{3}^{1},1,  1)$, $\Diag(1, 1, 1,\zeta_{3}^{1},  1)$ & 1 &  1 & Yes  \\ \hline
85 & $\mathbb Z / {2} \mathbb Z \times (\mathbb Z / {3} \mathbb Z)^{2}$ & $\Diag(-1, 1, 1,1,  1)$, $\Diag(1, 1, \zeta_{3}^{1},1,  1)$, $\Diag(1, 1, 1,\zeta_{3}^{1},  \zeta_{3}^{1})$ & 0 &  0 & Yes  \\ \hline
86 & $\mathbb Z / {2} \mathbb Z \times (\mathbb Z / {3} \mathbb Z)^{2}$ & $\Diag(-1, 1, 1,1,  1)$, $\Diag(1, 1, \zeta_{3}^{1},1,  1)$, $\Diag(1, 1, 1,\zeta_{3}^{1},  \zeta_{3}^{2})$ & 1 &  1 & Yes  \\ \hline
87 & $\mathbb Z / {2} \mathbb Z \times (\mathbb Z / {3} \mathbb Z)^{2}$ & $\Diag(-1, 1, 1,1,  1)$, $\Diag(1, 1, \zeta_{3}^{1},1,  \zeta_{3}^{1})$, $\Diag(1, 1, 1,\zeta_{3}^{1},  \zeta_{3}^{1})$ & 0 &  1 & No  \\ \hline
88 & $\mathbb Z / {2} \mathbb Z \times (\mathbb Z / {3} \mathbb Z)^{2}$ & $\Diag(-1, 1, 1,1,  1)$, $\Diag(1, 1, \zeta_{3}^{1},1,  \zeta_{3}^{2})$, $\Diag(1, 1, 1,\zeta_{3}^{1},  \zeta_{3}^{2})$ & 1 &  2 & No  \\ \hline
89 & $\mathbb Z / {8} \mathbb Z \times \mathbb Z / {3} \mathbb Z$ & $\Diag(\zeta_{8}^{1}, \zeta_{8}^{6}, \zeta_{8}^{4},1,  1)$, $\Diag(1, 1, 1,1,  \zeta_{3}^{1})$ & 0 &  0 & Yes  \\ \hline
90 & $\mathbb Z / {4} \mathbb Z \times \mathbb Z / {2} \mathbb Z \times \mathbb Z / {3} \mathbb Z$ & $\Diag(\zeta_{4}^{1}, -1, 1,1,  1)$, $\Diag(1, 1, 1,-1,  1)$, $\Diag(1, 1, 1,\zeta_{3}^{1},  \zeta_{3}^{1})$ & 0 &  0 & Yes  \\ \hline
91 & $\mathbb Z / {9} \mathbb Z \times \mathbb Z / {3} \mathbb Z$ & $\Diag(\zeta_{9}^{1}, \zeta_{9}^{7}, \zeta_{9}^{4},1,  1)$, $\Diag(1, 1, 1,\zeta_{3}^{1},  1)$ & 0 &  0 & Yes  \\ \hline
92 & $\mathbb Z / {9} \mathbb Z \times \mathbb Z / {3} \mathbb Z$ & $\Diag(\zeta_{9}^{1}, \zeta_{9}^{7}, \zeta_{9}^{4},1,  \zeta_{9}^{3})$, $\Diag(1, 1, 1,\zeta_{3}^{1},  1)$ & 0 &  0 & Yes  \\ \hline
93 & $\mathbb Z / {9} \mathbb Z \times \mathbb Z / {3} \mathbb Z$ & $\Diag(\zeta_{9}^{1}, \zeta_{9}^{7}, \zeta_{9}^{4},1,  \zeta_{9}^{6})$, $\Diag(1, 1, 1,\zeta_{3}^{1},  1)$ & 0 &  0 & Yes  \\ \hline
94 & $(\mathbb Z / {3} \mathbb Z)^{3}$ & $\Diag(1, \zeta_{3}^{1}, 1,1,  1)$, $\Diag(1, 1, \zeta_{3}^{1},1,  1)$, $\Diag(1, 1, 1,\zeta_{3}^{1},  1)$ & 0 &  0 & Yes  \\ \hline
95 & $(\mathbb Z / {3} \mathbb Z)^{3}$ & $\Diag(1, \zeta_{3}^{1}, 1,1,  1)$, $\Diag(1, 1, \zeta_{3}^{1},1,  1)$, $\Diag(1, 1, 1,\zeta_{3}^{1},  \zeta_{3}^{1})$ & 0 &  0 & Yes  \\ \hline
96 & $(\mathbb Z / {3} \mathbb Z)^{3}$ & $\Diag(1, \zeta_{3}^{1}, 1,1,  1)$, $\Diag(1, 1, \zeta_{3}^{1},1,  1)$, $\Diag(1, 1, 1,\zeta_{3}^{1},  \zeta_{3}^{2})$ & 1 &  1 & Yes  \\ \hline
97 & $(\mathbb Z / {3} \mathbb Z)^{3}$ & $\Diag(1, \zeta_{3}^{1}, 1,1,  1)$, $\Diag(1, 1, \zeta_{3}^{1},1,  \zeta_{3}^{1})$, $\Diag(1, 1, 1,\zeta_{3}^{1},  \zeta_{3}^{1})$ & 0 &  0 & Yes  \\ \hline
98 & $(\mathbb Z / {3} \mathbb Z)^{3}$ & $\Diag(1, \zeta_{3}^{1}, 1,1,  1)$, $\Diag(1, 1, \zeta_{3}^{1},1,  \zeta_{3}^{1})$, $\Diag(1, 1, 1,\zeta_{3}^{1},  \zeta_{3}^{2})$ & 0 &  0 & Yes  \\ \hline
99 & $(\mathbb Z / {3} \mathbb Z)^{3}$ & $\Diag(1, \zeta_{3}^{1}, 1,1,  1)$, $\Diag(1, 1, \zeta_{3}^{1},1,  \zeta_{3}^{2})$, $\Diag(1, 1, 1,\zeta_{3}^{1},  \zeta_{3}^{2})$ & 1 &  1 & Yes  \\ \hline
100 & $(\mathbb Z / {3} \mathbb Z)^{3}$ & $\Diag(1, \zeta_{3}^{1}, 1,1,  \zeta_{3}^{1})$, $\Diag(1, 1, \zeta_{3}^{1},1,  \zeta_{3}^{1})$, $\Diag(1, 1, 1,\zeta_{3}^{1},  \zeta_{3}^{1})$ & 0 &  1 & No  \\ \hline
101 & $(\mathbb Z / {3} \mathbb Z)^{3}$ & $\Diag(1, \zeta_{3}^{1}, 1,1,  \zeta_{3}^{1})$, $\Diag(1, 1, \zeta_{3}^{1},1,  \zeta_{3}^{1})$, $\Diag(1, 1, 1,\zeta_{3}^{1},  \zeta_{3}^{2})$ & 0 &  0 & Yes  \\ \hline
102 & $(\mathbb Z / {3} \mathbb Z)^{3}$ & $\Diag(1, \zeta_{3}^{1}, 1,1,  \zeta_{3}^{1})$, $\Diag(1, 1, \zeta_{3}^{1},1,  \zeta_{3}^{2})$, $\Diag(1, 1, 1,\zeta_{3}^{1},  \zeta_{3}^{2})$ & 0 &  1 & No  \\ \hline
103 & $(\mathbb Z / {3} \mathbb Z)^{3}$ & $\Diag(1, \zeta_{3}^{1}, 1,1,  \zeta_{3}^{2})$, $\Diag(1, 1, \zeta_{3}^{1},1,  \zeta_{3}^{2})$, $\Diag(1, 1, 1,\zeta_{3}^{1},  \zeta_{3}^{2})$ & 0 &  1 & No  \\ \hline
104 & $\mathbb Z / {4} \mathbb Z \times (\mathbb Z / {3} \mathbb Z)^{2}$ & $\Diag(\zeta_{4}^{1}, -1, 1,1,  1)$, $\Diag(1, 1, 1,\zeta_{3}^{1},  1)$, $\Diag(1, 1, 1,1,  \zeta_{3}^{1})$ & 0 &  0 & Yes  \\ \hline
105 & $(\mathbb Z / {2} \mathbb Z)^{2} \times (\mathbb Z / {3} \mathbb Z)^{2}$ & $\Diag(-1, 1, 1,1,  1)$, $\Diag(1, -1, 1,1,  1)$, $\Diag(1, \zeta_{3}^{1}, 1,\zeta_{3}^{1},  1)$, $\Diag(1, 1, 1,1,  \zeta_{3}^{1})$ & 0 &  0 & Yes  \\ \hline
106 & $\mathbb Z / {2} \mathbb Z \times (\mathbb Z / {3} \mathbb Z)^{3}$ & $\Diag(-1, 1, 1,1,  1)$, $\Diag(1, 1, \zeta_{3}^{1},1,  1)$, $\Diag(1, 1, 1,\zeta_{3}^{1},  1)$, $\Diag(1, 1, 1,1,  \zeta_{3}^{1})$ & 0 &  0 & Yes  \\ \hline
107 & $(\mathbb Z / {3} \mathbb Z)^{3} \times \mathbb Z / {2} \mathbb Z$ & $\Diag(1, \zeta_{3}^{1}, 1,1,  1)$, $\Diag(1, 1, -1,1,  1)$, $\Diag(1, 1, 1,\zeta_{3}^{1},  1)$, $\Diag(1, 1, 1,1,  \zeta_{3}^{1})$ & 0 &  0 & Yes  \\ \hline

\end{longtable}
\def\arraystretch{1}}

\printbibliography

\end{document}